\documentclass[12pt]{article}
\usepackage{latexsym,amsfonts,amsmath,graphics,amsthm}
\usepackage{epsfig}
\usepackage[usenames,dvipsnames]{color}

\newtheorem{theorem}{Theorem}
\newtheorem{lemma}{Lemma}
\newtheorem{corollary}{Corollary}
\newtheorem{proposition}{Proposition}
\newtheorem{conjecture}{Conjecture}
\newtheorem{definition}{Definition}

\newtheorem{remark}{Remark}

\newcommand{\satop}[2]{\stackrel{\scriptstyle{#1}}{\scriptstyle{#2}}}

\newcommand{\bsw}{\boldsymbol{w}}
\newcommand{\bsx}{\boldsymbol{x}}

\newcommand{\bsz}{\boldsymbol{z}}

\newcommand{\bsy}{\boldsymbol{y}}

\newcommand{\bszero}{\boldsymbol{0}}

\newcommand{\NN}{\mathbb{N}}

\newcommand{\RR}{\mathbb{R}}

\newcommand{\HH}{\mathcal{H}}

\DeclareMathOperator{\dd}{\mathrm{d}}
\DeclareMathOperator{\dist}{dist}

\newcommand{\DEF}{{:=}}
\newcommand{\FED}{{=:}}
\newcommand{\PT}[1]{\mathbf{#1}}

\allowdisplaybreaks

\begin{document}

\title{\scshape Quasi-Monte Carlo rules for numerical integration over the unit sphere $\mathbb{S}^2$}

\author{Johann S. Brauchart\thanks{\noindent The author is supported by an {\sc APART}-Fellowship of the Austrian Academy of Sciences.} and Josef Dick\thanks{The author is supported by an Australian Research Council Queen Elizabeth 2 Fellowship.} \\ School of Mathematics and Statistics,\\
University of New South Wales, \\
Sydney, NSW, 2052, Australia \\ \texttt{j.brauchart@unsw.edu.au}, \\ \texttt{josef.dick@unsw.edu.au}}

\date{\today}
\maketitle

\begin{abstract}
We study numerical integration on the unit sphere $\mathbb{S}^2 \subseteq
\mathbb{R}^3$ using equal weight quadrature rules, where the weights
are such that constant functions are integrated exactly. 

The quadrature points are constructed by lifting a $(0,m,2)$-net given in the 
unit square $[0,1]^2$ to the sphere $\mathbb{S}^2$ by means of an area preserving map. 
A similar approach has previously been suggested by Cui and Freeden [SIAM J. Sci. Comput. 18 (1997), no. 2].

%
We prove three results. The first one is that the construction is
(almost) optimal with respect to discrepancies based on spherical
rectangles. Further we prove that the point set is asymptotically
uniformly distributed on $\mathbb{S}^2$. And finally, we prove an
upper bound on the spherical cap $L_2$-discrepancy of order
$N^{-1/2} (\log N)^{1/2}$ (where $N$ denotes the number of points). This improves upon the bound on the spherical cap $L_2$-discrepancy of the construction by Lubotzky, Phillips and Sarnak [Comm. Pure Appl. Math. 39 (1986)] by a factor of $\sqrt{\log N}$.

Numerical results suggest that the $(0,m,2)$-nets lifted to the sphere
$\mathbb{S}^2$ have spherical cap $L_2$-discrepancy converging with
the optimal order of $N^{-3/4}$.
\end{abstract}

\section{Introduction}

We consider the unit sphere $\mathbb{S}^2 = \{\boldsymbol{z} =
(z_1,z_2,z_3) \in \mathbb{R}^3: \|\boldsymbol{z}\| =
\sqrt{z_1^2+z_2^2+z_3^2} = 1\}$. Let $f:\mathbb{S}^2\to\mathbb{R}$ be integrable. Then we estimate the integral $\int_{\mathbb{S}^2} f \, \mathrm{d}\sigma_2$, where $\sigma_2$ is the normalized Lebesgue surface area measure (that is $\int_{\mathbb{S}^2} \mathrm{d} \sigma_2 = 1$), by a quasi-Monte Carlo type rule
\begin{equation*}
Q_N(f) = \frac{1}{N}\sum_{k=0}^{N-1} f(\bsz_k),
\end{equation*}
where $Z_N = \{\bsz_0,\ldots, \bsz_{N-1}\} \subseteq \mathbb{S}^2$ are the quadrature points on the sphere. Since the surface area measure is normalized to $1$, it follows that we have $Q_N(f) = \int_{\mathbb{S}^2} f \,\mathrm{d}\sigma_2$ for every constant function $f$. In the following we review known results to put the result of this paper into context. Although some results are known for spheres $\mathbb{S}^d$ of dimension $d \ge 2$, we only state them for the sphere $\mathbb{S}^2$ since this paper only deals with the $2$-sphere.

Using Stolarsky's invariance principle \cite{St1973} (also see \cite{BrDi2011_pre} and \cite{BrWo20xx}), it follows that the worst-case error for numerical integration in a certain reproducing kernel Hilbert space is given by the spherical cap $L_2$-discrepancy of the quadrature points. To obtain quadrature points we use a transformation from $[0,1]^2$ to $\mathbb{S}^2$ which preserves area. 
Specifically, we use the transformation $\Phi:[0,1]^2\to\mathbb{S}^2$ with
\begin{equation*}
\Phi(y_1,y_2) = \left(2 \cos (2\pi y_1) \sqrt{y_2- y_2^2}, 2 \sin (2\pi y_1) \sqrt{y_2- y_2^2}, 1-2 y_2\right).
\end{equation*}
The function $\Phi$ maps axis-parallel rectangles in the unit square to zonal spherical rectangles of equal area. It is natural then that the discrepancy on the sphere with respect to spherical rectangles is the same as the discrepancy with respect to rectangles in $[0,1]^2$. Since point sets $Z_N$ for which a discrepancy based on $K$-regular test sets ($K$ fixed, see Sj{\"o}gren~\cite{Sj1972}) converges to zero as $N \to \infty$ are uniformly distributed over the sphere, we show that the digital nets lifted to the sphere via $\Phi$ are also uniformly distributed.

Furthermore, we study the spherical cap $L_2$-discrepancy of point sets obtained in this way. We prove that the order of convergence for $(0,m,2)$-nets lifted to the sphere via $\Phi$ is $\mathcal{O}(N^{-1/2} (\log N)^{1/2})$. This improves upon the bound on the spherical cap $L_2$-discrepancy in \cite{LuPhSa1986} of $\mathcal{O}(N^{-1/2} \log N)$ by a factor of $\sqrt{\log N}$. On the other hand, the optimal order of convergence is $\mathcal{O}(N^{-3/4})$. But numerical experiments do suggest that the $(0,m,2)$-nets lifted to the sphere via $\Phi$ achieve optimal convergence rate. We conjecture that this order of convergence is the correct one.

Due to the difficulties of having a satisfactory notion of 'bounded variation' on $\mathbb{S}^2$ there is no Koksma-Hlawka inequality on $\mathbb{S}^2$ per se. However, the concept of {\em uniform distribution of a sequence of $N$-point configurations with respect to every function in a function space}, say Sobolev spaces over $\mathbb{S}^2$, can be used as quality criterion. 
For example, Cui and Freeden~\cite{CuFr1997} introduced the concept of a generalized discrepancy on $\mathbb{S}^2$ which involves pseudodifferential operators. Using the operator $\mathbf{D} = ( - 2 \Delta^* )^{1/2} ( - \Delta^* + 1 / 4 )^{1/4}$ ($\Delta^*$ is the Beltrami operator on $\mathbb{S}^2$) of order $3/2$ they arrived at
\begin{equation*}
\left| \frac{1}{N} \sum_{k = 1}^N f(\bsx_k) - \int f \dd \sigma_2 \right| \leq \sqrt{6} \, \mathrm{D}(\{ \bsx_1, \dots, \bsx_N \}; \mathbf{D} ) \, \| f \|_{\HH^{3/2}},
\end{equation*}
where $f$ is from the Sobolev space $\HH^{3/2}(\mathbb{S}^2)$. In this notion, a sequence $\{Z_N\}$ of $N$-point configurations is called {\em $\mathbf{D}$-equidistributed in $\HH^{3/2}(\mathbb{S}^2)$} if $\lim_{N\to\infty} \mathrm{D}(\{ \bsz_1, \dots, \bsz_N \}; \mathbf{D} ) = 0$. The generalized discrepancy associated with $\mathbf{D}$ can be easily computed by way of
\begin{equation} \label{eq:sum.log.discr}
4 \pi \left[ \mathrm{D}(\{ \bsz_1, \dots, \bsz_N \}; \mathbf{D} ) \right]^2 = 1 - \frac{1}{N^2} \sum_{k, \ell = 1}^N \log \left( 1 + \left\| \bsz_\ell - \bsz_k \right\| / 2 \right)^2.
\end{equation}
Sloan and Womersley~\cite{SlWo2004} showed that $[\mathrm{D}(\{
\bsz_1, \dots, \bsz_N \}; \mathbf{D} )]^2$ has a natural
interpretation as the worst-case error for the equally weighted
quadrature rule $Q_N$ associated with the points $\bsz_1, \dots,
\bsz_N$ for functions $f$ from the unit ball in
$\HH^{3/2}(\mathbb{S}^2)$ provided with the reproducing kernel
$K(\bsy, \bsz) = 2 [ 1 - \log( 1 + \| \bsy - \bsz \| / 2 ) ]$. In
\cite{BrWo20xx} this approach is followed further, yielding an even
simpler notion of discrepancy (see
Section~\ref{sec:worst-case-error}), namely
\begin{equation} \label{eq:sum.dist.discr}
\left[ \mathrm{D}(\{ \bsz_1, \dots, \bsz_N \} ) \right]^2 = \frac{4}{3} - \frac{1}{N^2} \sum_{k, \ell = 1}^N \left\| \bsz_\ell - \bsz_k \right\|
\end{equation}
also used in the setting of the Sobolev space $\HH^{3/2}(\mathbb{S}^2)$ now provided with the reproducing kernel $K(\bsy, \bsz) = (8 / 3) - \| \bsy - \bsz \|$. Low-discrepancy configurations in the above contexts can be found by maximizing the respective double sums in \eqref{eq:sum.log.discr} and \eqref{eq:sum.dist.discr}. This leads into the realm of the discrete (minimum) energy problem on the sphere where points are thought to interact according to a Riesz $s$-potential $1/r^s$ ($s>0$) or logarithmic potential $\log(1/r)$ ($s=0$) and $r$ is the Euclidean distance in the ambient space. It is known that the minimizer $Z_N$ ($N\geq2$) of the associated $s$-energy functionals form a sequence which is asymptotically uniformly distributed for each fixed $s\geq0$. We refer the interested reader to \cite{BaBr2009, BeCa2009, BeClDu2004, BoHaSa2008, Br2008, CoKu2007, HaSa2004, HaSa2005, Wa1990, Wa1992b, Wa1992}. We remark that Sun and Chen~\cite{SuCh2008} employ 'spherical basis functions' (as a counter part to radial basis function on spheres) to investigate uniform distribution on spheres. They also show that the minimizers of the functionals induced by a spherical basis function are asymptotically uniformly distributed.

Let $C(\bsz, t) \DEF \{ \bsy \in \mathbb{S}^2 : \langle \bsy, \bsz \rangle \leq t \}$ be a spherical cap and let $\mathcal{C} \DEF \{C(\bsz,t): \bsz \in \mathbb{S}^2, -1 \le t \le 1\}$ denote the set of all spherical caps. In \cite{St1973}, Stolarsky established a beautiful invariance principle on the $2$-sphere (and higher-dimensional spheres),
\begin{equation} \label{eq:stolarsky.inv.prncpl}
\frac{1}{N^2} \sum_{k, \ell = 1}^N \left\| \bsz_\ell - \bsz_k \right\| + 4 \left[ D_2( Z_N; \mathbb{S}^2, \mathcal{C} ) \right]^2 = \int_{\mathbb{S}^2} \int_{\mathbb{S}^2} \left\| \bsy - \bsz \right\| \dd \sigma_2(\bsy) \dd \sigma_2(\bsz),
\end{equation}
connecting the sum of distances, the $L_2$-discrepancy (with respect to spherical caps $C(\bsz, t)$),
\begin{equation*}
D_2( Z_N; \mathbb{S}^2, \mathcal{C} ) \DEF \left[ \int_{-1}^1 \int_{\mathbb{S}^2} \left| \frac{\left| Z_N \cap C( \bsz, t) \right|}{N} - \sigma_2( C( \bsz, t ) ) \right|^2 \dd \sigma_2( \bsz ) \dd t \right]^{1/2},
\end{equation*}
and the distance integral. We observe that, on $\mathbb{S}^2$, the discrepancy in \eqref{eq:sum.dist.discr} is essentially (up to a factor $2$) the $L_2$-discrepancy. Originally, Stolarsky used his invariance principle and sharp result for discrepancy estimates by Schmidt~\cite{Sch1969} to establish bounds for the sum of distances. Beck~\cite{Be1984} then improved Stolarsky's lower bound, finally arriving at
\begin{equation} \label{eq:Stolarsky.Beck.bounds}
c \, N^{-3/2} \leq \int \int \left\| \bsx - \bsy \right\| \dd \sigma_2(\bsx) \dd \sigma_2(\bsy) - \frac{1}{N^2} \sum_{k, \ell = 1}^N \left\| \bsz_\ell - \bsz_k \right\| \leq C \, N^{-3/2}
\end{equation}
for some universal positive constants $c$ and $C$ independent of $N$. Consequently, relations \eqref{eq:Stolarsky.Beck.bounds} yield lower and upper bound for the $L_2$-discrepancy by means of the invariance principle which are sharp with respect to order of $N$:
\begin{equation*}
c^\prime \, N^{-3 / 4} \leq D_2( Z_N^*;\mathbb{S}^2, \mathcal{C} ) \leq C^\prime \, N^{-3 / 4}
\end{equation*}
for a sequence of optimal $L_2$-discrepancy $N$-point configurations $Z_N^*$ on $\mathbb{S}^2$. Observe, that optimal $L_2$-discrepancy and optimal spherical cap discrepancy configurations satisfy estimates with the same order $N^{-3/4}$ (apart from a possible $\sqrt{\log N}$ term), see \cite{Be1984}.

\subsection{Quasi-Monte Carlo rules in the unit square}

Quasi-Monte Carlo algorithms $\widehat{I}(f) = \frac{1}{N}
\sum_{n=0}^{N-1} f(\bsx_n)$ are used to approximate integrals $I(f)
= \int_{[0,1]^2} f(\bsx) \,\mathrm{d} \bsx$. The crux of the method
is to choose the quadrature points $\bsx_0,\ldots, \bsx_{N-1}$ as
uniformly distributed as possible. The difference to Monte Carlo is
the method by which the sample points $\bsx_0,\ldots, \bsx_{N-1} \in
[0,1)^2$ are chosen. The aim of QMC is to chose those points such
that the integration error
\begin{equation*}
\left|\int_{[0,1]^2} f(\bsx) \,\mathrm{d}\bsx - \frac{1}{N}
\sum_{n=0}^{N-1} f(\bsx_n) \right|
\end{equation*}
achieves the (almost) optimal rate of convergence as $N \to\infty$ for certain classes of
functions $f:[0,1]^2 \to \mathbb{R}$. For instance, for the family of
functions $f$ with bounded variation in the sense of
Hardy and Krause (for which we write $\|f\|_{\mathrm{HK}} < \infty$)
it is known that the best rate of convergence for the worst case
error is
\begin{equation*}
e = \sup_{f, \|f\|_{\mathrm{HK}} <\infty} \left|\int_{[0,1]^2}
f(\bsx) \,\mathrm{d} \bsx - \frac{1}{N} \sum_{n=0}^{N-1} f(\bsx_n)
\right| \asymp N^{-1+\varepsilon} \quad \mbox{for all } \varepsilon
> 0.
\end{equation*}
More precisely, there are constants $c,C > 0$ such that 
\begin{equation*}
c N^{-1} \sqrt{\log N} \le e \le C N^{-1} \sqrt{\log N},
\end{equation*}
see \cite{DiPi2010}.


There is an explicit construction of the sample points
$\bsx_0,\ldots, \bsx_{N-1}$ for which the optimal rate of
convergence is achieved. One criterion for how uniformly a set of points $P_N = \{\bsx_0,\ldots,
\bsx_{N-1}\}$ is distributed in the unit square is the star discrepancy
\begin{equation*}
D^\ast(P_N; [0,1]^2, \mathcal{R}^\ast) = \sup_{\bsy \in [0,1]^2} \left| \delta_{P_N}(\bsy) \right|, \qquad \delta_{P_N}(\bsy) = \frac{1}{N} \sum_{n=0}^{N-1} 1_{\bsx_i \in
[\bszero,\bsy)} - \mathrm{Area}([\bszero,\bsy)),
\end{equation*}
where $[\bszero,\bsy)=\prod_{i=1}^s [0,y_i)$ with $\bsy=(y_1, y_2) \in [0,1]^2$, $\mathcal{R}^\ast = \{[\bszero,\bsy) : \bsy \in [0,1]^2\}$, $\mathrm{Area}([\bszero,\bsy))= y_1 y_2$ is the area of $[\bszero,\bsy)$ and 
\begin{equation*}
1_{\bsx_i \in [\bszero,\bsy)} = 
\begin{cases}
1 & \text{if $\bsx_i \in [\bszero,\bsy)$,} \\
0 & \text{otherwise.}
\end{cases}
\end{equation*}
The quantity $\delta_{P_N}(\bsy)$ is called the local discrepancy (of $P_N$).

The connection between this criterion and the integration error is given by
the Koksma-Hlawka inequality
\begin{equation*}
\left|\int_{[0,1]^2} f(\bsx) \,\mathrm{d} \bsx - \frac{1}{N}
\sum_{n=0}^{N-1} f(\bsx_n) \right| \le D^\ast(P_N; [0,1]^2, \mathcal{R}^\ast)
\|f\|_{\mathrm{HK}}.
\end{equation*}

Informally, a sequence of points $\bsx_0,\bsx_1,\ldots \in [0,1)^s$ is called a low-discrepancy sequence, if
\begin{equation*}
D^\ast(\{\bsx_0,\ldots, \bsx_{N-1}\}; [0,1]^2, \mathcal{R}^\ast) = \mathcal{O}(N^{-1} (\log N)^2) \quad\mbox{as } N \to \infty.
\end{equation*}
Notice that for such a point sequence, the Koksma-Hlawka inequality implies the optimal rate of convergence
of the integration error (apart from the power of the $\log N$ factor), since for a given integrand $f$ the variation $\|f\|_{\mathrm{HK}}$ does not depend on $P_N$ and $N$ at all. 

The concept of digital nets introduced by Niederreiter~\cite{Ni1988} provides the to date most efficient method to explicitly construct points sets $P_N = \{\bsx_0,\ldots, \bsx_{N-1} \} \in [0,1)^2$ with small discrepancy, that is
\begin{equation*}
D^\ast(P_N; [0,1]^2, \mathcal{R}^\ast) \le C N^{-1} \log N.
\end{equation*}
They are introduced in the next section.


\subsection{Nets and sequences in the unit square}

In this section we give a brief overview of (digital) $(0,m,2)$-nets
and (digital) $(0,2)$-sequences. For a comprehensive introduction
see \cite{DiPi2010}.

The aim is to construct a point set $P_N = \{\bsx_0,\ldots,
\bsx_{N-1}\}$ such that the star discrepancy satisfies
$D^\ast(P_N;[0,1]^2, \mathcal{R}^\ast) \le C N^{-1} \log N$. To do so, we discretize
the problem by choosing the point set $P_N$ such that the local
discrepancy $\delta_{P_N}(\bsy) = 0$ for certain $\bsy \in [0,1]^2$
(those $\bsy$ in turn are chosen such that the star discrepancy of
$P_N$ is small, as we explain below).

It turns out that, when one chooses a base $b \geq 2$ and $N = b^m$,
then for every natural number $m$ there exist point sets
$P_{b^m}=\{\bsx_0,\ldots, \bsx_{b^m-1}\}$ such that
$\delta_{P_{b^m}}(\bsy) = 0$ for all $\bsy=(y_1, y_2)$ of the
form
\begin{equation*}
y_i = a_i / b^{d_i} \quad \text{for $i = 1,2$,}
\end{equation*}
where $0 < a_i \le b^{d_i}$ is an integer and $d_1 + d_2
\le m$ with $d_1, d_2 \ge 0$. A point set
$P_N$ which satisfies this property is called a $(0,m,2)$-net in
base $b$. An equivalent description of $(0,m,2)$-nets in base $b$ is
given in the following definition.

\begin{definition}
Let $b \ge 2$ and $m \ge 1$ be integers. A point set
$P_{b^m} \subseteq [0,1)^2$ consisting of $b^m$ points is
called a $(0,m,2)$-net in base $b$, if for all nonnegative integers
$d_1,d_2$ with $d_1 + d_2 = m$, the elementary
interval
\begin{equation*}
\prod_{i=1}^2 \left[\frac{a_i}{b^{d_i}},
\frac{a_i+1}{b^{d_i}}\right)
\end{equation*}
contains exactly $1$ point of $P_{b^m}$ for all integers $0
\le a_i < b^{d_i}$.
\end{definition}
It is also possible to construct nested $(0,m,2)$-nets, thereby
obtaining an infinite sequence of points.
\begin{definition}
Let $b \ge 2$ be an integer. A sequence
$\bsx_0,\bsx_1,\ldots \in [0,1)^2$ is called a $(0,2)$-sequence in
base $b$, if for all $m > 0$ and for all $k \ge 0$, the point set
$\bsx_{k b^m}, \bsx_{kb^m + 1}, \ldots, \bsx_{(k+1)b^m-1}$ is a
$(0,m,2)$-net in base $b$.
\end{definition}

It can be shown that a $(0,m,2)$-net in base $b$ satisfies
\begin{equation*}
D^\ast(P_N; [0,1]^2, \mathcal{R}^\ast) \le C_{b} \frac{m}{b^{m-1}},
\end{equation*}
and the first $N$ points $\bsx_0,\ldots, \bsx_{N-1}$ of a $(0,2)$-sequence in base $b$ satisfy
\begin{equation*}
D^\ast(\{\bsx_0,\ldots, \bsx_{N-1}\}; [0,1]^2, \mathcal{R}^\ast) \le C_{b} \frac{(\log N)^2}{N} \quad \mbox{for all } N \ge 1,
\end{equation*}
where $C_{b} > 0$ depends on $b$ but not on $m$ and $N$. See \cite{DiPi2010,Ni1992} for details.

Explicit constructions of $(0,m,2)$-nets and $(0,2)$-sequences can
be obtained using the digital construction scheme. Such point sets
are then called digital nets (or digital $(0,m,2)$-nets if the point
set is a $(0,m,2)$-net) or digital sequences (or digital
$(0,2)$-sequence if the sequence is a $(0,2)$-sequence).

To describe the digital construction scheme, let $b$ be a prime
number and let $\mathbb{Z}_b$ be the finite field of order $b$ (a
prime power and the finite field $\mathbb{F}_b$ could be used as
well). Let $C_1, C_2 \in \mathbb{Z}_b^{m \times m}$ be $2$
matrices of size $m \times m$ with elements in $\mathbb{Z}_b$. The $i$th coordinate $x_{n,i}$ of the $n$th point $\bsx_n=(x_{n,1}, x_{n,2})$ of the digital net is
obtained in the following way. For $0 \le n < b^m$ let $n = n_0 +
n_1 b + \cdots + n_{m-1} b^{m-1}$ be the base $b$ representation of
$n$. Let $\vec{n} = (n_0,\ldots, n_{m-1})^\top \in \mathbb{Z}_b^m$
denote the vector of digits of $n$. Then let
\begin{equation*}
\vec{y}_{n,i} = C_i \vec{n}.
\end{equation*}
For $\vec{y}_{n,i} = (y_{n,i,1},\ldots, y_{n,i,m})^\top \in
\mathbb{Z}_b^{m}$ we set
\begin{equation*}
x_{n,i} = \frac{y_{n,i,1}}{b} + \cdots + \frac{y_{n,i,m}}{b^{m}}.
\end{equation*}
To construct digital sequences, the generating matrices $C_1,
C_2$ are of size $\infty \times \infty$.

The search for $(0,m,2)$-nets and $(0,2)$-sequences has now been
reduced to finding suitable matrices $C_1, C_2$. Explicit
constructions of such matrices are available and where introduced
(in chronological order) by Sobol~\cite{So1967},
Faure~\cite{Fa1982}, Niederreiter~\cite{Ni1988}, Niederreiter and
Xing~\cite{NiXi1996,NiXi1995,XiNi1995}, as well as others. For instance, to obtain a digital $(0,m,2)$-net over $\mathbb{Z}_2$ one can choose
\[
C_{1} = \left(
\begin{array}{ccccc}
  1      & 0 & \ldots & 0 & 0 \\
  0      & 1 & \ddots &   & 0 \\
\vdots & \ddots & \ddots & \ddots & \vdots \\
0 &  & \ddots & 1 & 0 \\
0 & 0 & \ldots & 0 & 1
\end{array} \right) \mbox{ and }\;\;
C_{2} = \left(
\begin{array}{ccccc}
  {0 \choose 0} & {1 \choose 0} & \ldots & \ldots & {m-1 \choose 0} \\
  0             & {1 \choose 1} & \ldots & \ldots & {m-1 \choose 1} \\
\vdots & \ddots & \ddots & & \vdots \\
  0             & \ldots & 0 & {m-2 \choose m-2}      & {m-1 \choose m-2}\\
  0             & \ldots    & \ldots     & 0      & {m-1 \choose m-1}
\end{array} \right),
\]
where the binomial coefficients are taken modulo $2$. Notice that the matrices $C_1, C_2$ can be extended to $\infty \times \infty$ matrices yielding generating matrices for a digital $(0,2)$-sequence over $\mathbb{Z}_2$. This example was first provided by Sobol$'$ in \cite{So1967}.

\section{Spherical rectangle discrepancy of point sets on the unit sphere $\mathbb{S}^2$}

We introduce a transformation to lift $(0,m,2)$-nets to the sphere $\mathbb{S}^2$. To do so, we represent each point on the sphere $\mathbb{S}^2$ by using scaled spherical coordinates. We define $T: [0,1) \times [0,1]  \mapsto  \mathbb{S}^2$ by
\begin{equation*}
T(\theta,\phi)  =  \left(\cos (2 \pi \theta) \sin (\pi \phi),\sin (2 \pi \theta) \sin (\pi \phi), \cos  (\pi \phi) \right).
\end{equation*}

Let $0 \le \theta_1 \le \theta \le \theta_2 \le 1$ and $0 \le \phi_1
\le \phi_2 \le 1$. The part of the sphere
\begin{equation*}
\Omega_{\theta_1,\theta_2, \phi_1,\phi_2} = \{T(\theta,\phi) \in \mathbb{S}^2:
\theta_1 \le \theta < \theta_2, \phi_1 \le \phi < \phi_2\}
\end{equation*}
has area
\begin{equation*}
\mathrm{Area}(\Omega_{\theta_1,\theta_2, \phi_1,\phi_2}) = 2 \pi
(\theta_2-\theta_1) \left[\cos (\pi \phi_1) - \cos (\pi \phi_2) \right].
\end{equation*}
Let
\begin{align*}
\Gamma_{\theta_1,\theta_2,\phi_1,\phi_2} &:= \frac{\mathrm{Area}(\Omega_{\theta_1,\theta_2,\phi_1,\phi_2})}{\mathrm{Area}(\mathbb{S}^2)} = (\theta_2-\theta_1) \frac{\cos (\pi \phi_1) - \cos (\pi \phi_2)}{2},
\end{align*}
be the area of $\Omega_{\theta_1,\theta_2,\phi_1,\phi_2}$ normalized with respect to the surface area of the unit sphere. Therefore we have $\Gamma_{0,1,0,1} = 1$. We can  now define a discrepancy measure of point sets on the sphere with respect to sets $\Omega_{\theta_1,\theta_2,\phi_1,\phi_2}$.

\begin{definition}\label{def_disc_sphere}
Let $Z_N = \{\bsz_0,\ldots, \bsz_{N-1}\} \subseteq \mathbb{S}^2$. Then the extreme spherical rectangle discrepancy of $Z_N$ on $\mathbb{S}^2$ with respect to $\Omega = \{\Omega_{\theta_1,\theta_2,\phi_1,\phi_2}: 0 \le \theta_1 < \theta_2 \le 1, 0 \le \phi_1 < \phi_2 \le 1\}$ is defined by
\begin{equation*}
D_N(Z_N;\mathbb{S}^2,\Omega) = \sup_{\satop{0 \le \theta_1 < \theta_2 \le 1}{0 \le \phi_1 < \phi_2 \le 1}} \left|\frac{1}{N} \sum_{n=0}^{N-1} 1_{\bsz_n \in \Omega_{\theta_1,\theta_2,\phi_1,\phi_2}} - \Gamma_{\theta_1,\theta_2,\phi_1,\phi_2} \right|,
\end{equation*}
where 
\begin{equation*}
1_{\bsz_n \in \Omega_{\theta_1,\theta_2,\phi_1,\phi_2}} = 
\begin{cases}
1 & \text{if $\bsz_n \in \Omega_{\theta_1,\theta_2,\phi_1,\phi_2}$} \\
0 & \text{otherwise.}
\end{cases}
\end{equation*}
%

The spherical rectangle star-discrepancy of $Z_N$ on $\mathbb{S}^2$ with respect to $\Omega^\ast = \{\Omega_{0,\theta,0,\phi}:0\le \theta \le 1, 0 \le \phi \le 1\}$ is defined by
\begin{equation*}
D^\ast_N(Z_N;\mathbb{S}^2,\Omega^\ast) = \sup_{\satop{0 \le \theta \le 1}{0 \le \phi \le 1}} \left|\frac{1}{N} \sum_{n=0}^{N-1} 1_{\bsz_n \in \Omega_{0,\theta,0,\phi}} - \Gamma_{0,\theta,0,\phi} \right|.
\end{equation*}
\end{definition}

The expression $\frac{1}{N} \sum_{n=0}^{N-1} 1_{\bsz \in
\Omega_{\theta_1,\theta_2,\phi_1,\phi_2}}$ denotes the proportion of
points of $Z_N$ in
$\Omega_{\theta_1,\theta_2,\phi_1,\phi_2}$. Note that the set
$\Omega_{\theta_1,\theta_2,\phi_1,\phi_2}$ always includes the North
Pole $(0,0,1)$ and never the South Pole $(0,0,-1)$.

Further, the discrepancy measure includes the discrepancy of
spherical caps centered at the North Pole $(0,0,1)$ and South Pole
$(0,0,-1)$. The spherical cap at the North Pole is obtained by using
$\Omega_{0,1,0,\phi}$. Let $\phi > 0$ be small. Then
$\frac{1}{N}\sum_{n=0}^{N-1} 1_{\bsz_n \in \Omega_{0,1,0,\phi}}$ is
the proportion of points in the spherical cap $\Omega_{0,1,0,\phi}$
and $4\pi \Gamma_{0,1,0,\phi}$ is the area of the spherical cap.
Hence 
\begin{equation*}
\left|\frac{1}{N} \sum_{n=0}^{N-1} 1_{\bsz_n \in
\Omega_{0,1,0,\phi}} - \Gamma_{0,1,0,\phi}\right|
\end{equation*}
measures the discrepancy at the North Pole. The South Pole is included in the
following way. Let $\phi < 1$ be close to $1$. Then $\mathbb{S}^2
\setminus \Omega_{0,1,0,\phi}$ is the spherical cap centered at the
South Pole. Then 
\begin{equation*}
\frac{1}{N} \sum_{n=0}^{N-1} 1_{\bsz_n \in
\mathbb{S}^2 \setminus \Omega_{0,1,0,\phi}} = 1-\frac{1}{N}
\sum_{n=0}^{N-1} 1_{\bsz_n \in \Omega_{0,1,0,\phi}}
\end{equation*}
is the proportion of points in the spherical cap centered at the South
Pole. Further, $4\pi (1-\Gamma_{0,1,0,\phi})$ is the area of the
spherical cap $\mathbb{S}^2 \setminus \Omega_{0,1,0,\phi}$. Hence
\begin{equation*}
\left|\frac{1}{N} \sum_{n=0}^{N-1} 1_{\bsz_n \in \Omega_{0,1,0,\phi}} - \Gamma_{0,1,0,\phi} \right| = \left|\frac{1}{N} \sum_{n=0}^{N-1} 1_{\bsz_n \in \mathbb{S}^2 \setminus \Omega_{0,1,0,\phi}} - \frac{\mathrm{Area}(\mathbb{S}^2\setminus
\Omega_{0,1,0,\phi})}{\mathrm{Area}(\mathbb{S}^2)}\right|
\end{equation*}
measures the discrepancy at the South Pole.

In the following we construct point sets $Z_N = \{\bsz_0,\ldots,
\bsz_{N-1}\}$ on the sphere $\mathbb{S}^2$ such that
$D_N(Z_N;\mathbb{S}^2,\Omega)$ is small. This can be done by relating the spherical rectangle discrepancy to an analogous discrepancy on $[0,1]^2$.

\begin{definition}\label{def_disc_square}
Let $P_N = \{\bsx_0,\ldots, \bsx_{N-1}\} \subseteq [0,1)^2$. Then the extreme discrepancy of $P_N$ on $[0,1)^2$ with respect to $\mathcal{R} = \{[a_1,a_2) \times [c_1,c_2): 0 \le a_1 < a_2 \le 1, 0 \le c_1 < c_2 \le 1\}$ is defined by
\begin{equation*}
D(P_N;[0,1)^2,\mathcal{R}) = \sup_{\satop{0 \le a_1 < a_2 \le 1}{0 \le c_1 < c_2 \le 1}} \left|\frac{1}{N} \sum_{n=0}^{N-1} 1_{\bsx_n \in [a_1,a_2) \times [c_1,c_2)} - (a_2-a_1)(c_2-c_1) \right|,
\end{equation*}
where 
\begin{equation*}
1_{\bsx_n \in [a_1,a_2) \times [c_1,c_2)} = 
\begin{cases}
1 & \text{if $\bsx_n \in [a_1,a_2) \times [c_1,c_2)$,} \\
0 & \text{otherwise.}
\end{cases}
\end{equation*}

The star-discrepancy of $P_N$ on $[0,1)^2$ with respect to $\mathcal{R}^\ast = \{[0,a)\times [0,c):0 \le a \le 1, 0 \le c \le 1\}$ is defined by
\begin{equation*}
D^\ast(P_N;[0,1)^2,\mathcal{R}^\ast) = \sup_{\satop{0 \le a \le 1}{0 \le c \le 1}} \left|\frac{1}{N} \sum_{n=0}^{N-1} 1_{\bsx_n \in [0,a)\times [0,c)} - ac \right|.
\end{equation*}
\end{definition}

Let $P_N$ be an arbitrary point set in $[0,1)^2$. It is known, see for instance \cite[Chapter~3]{DiPi2010}, that
\begin{equation*}
D^\ast(P_N; [0,1)^2, \mathcal{R}^\ast) \le D_N(P_N; [0,1)^2, \mathcal{R}) \le 4 D^\ast(P_N; [0,1)^2, \mathcal{R}^\ast).
\end{equation*}

The following lemma establishes a connection between the discrepancies on $[0,1)^2$ and $\mathbb{S}^2$ defined above.
\begin{lemma}\label{lem_disc}
Let $P_N = \{\bsx_0,\ldots, \bsx_{N-1}\} \in (0,1)^2$ and
$\bsx_n = (x_{n,1}, x_{n,2})$ for $0 \le n < N$. Let
\begin{equation*}
\bsz_n = T\left(x_{n,1}, \frac{\arccos (1-2x_{n,2})}{\pi}\right) \qquad \text{for $0 \le n < N$} 
\end{equation*}
and $Z_N=\{\bsz_0,\ldots, \bsz_{N-1}\}$. Then $\bsz_n \in \mathbb{S}^2$ for $0 \le n < N$ and
\begin{equation*}
D_N(Z_N; \mathbb{S}^2, \Omega) = D_N(P_N; [0,1)^2, \mathcal{R})
\end{equation*}
and
\begin{equation*}
D^\ast_N(Z_N; \mathbb{S}^2, \Omega^\ast) = D_N(P_N; [0,1)^2, \mathcal{R}^\ast).
\end{equation*}
\end{lemma}

\begin{proof}
It is trivial to check that $\|\bsz_n\|=1$ for all $n$. To show the second part, we set $\theta_1 = a_1$, $\theta_2 = a_2$, $\phi_1 = [ \arccos (1-2 c_1) ] / \pi$ and $\phi_2 = [\arccos (1-2 c_2) ] / \pi$. Then we have
\begin{equation}\label{eq_area}
\Gamma_{\theta_1,\theta_2,\phi_1,\phi_2} = (a_2-a_1) \frac{1-2 c_1 -
(1-2 c_2)}{2} = (a_2-a_1)(c_2-c_1).
\end{equation}
Note that the points $(0,0,1)$ and $(0,0,-1)$ are excluded from the
set $Z_N$ since we assume that $P_N \subseteq
(0,1)^2$. The mapping from $P_N$ to $Z_N$ is
therefore bijective. Therefore we have
\begin{equation*}
1_{\bsz_n\in \Omega_{\theta_1,\theta_2,\phi_1,\phi_2}} = 1_{(y_{n,1},(\arccos(1-2y_{n,2}))/\pi) \in [\theta_1,\theta_2) \times [\phi_1,\phi_2)} = 1_{(y_{n,1},y_{n,2}) \in [a_1,a_2) \times [c_1, c_2)}.
\end{equation*}
Thus the remaining results follow from
Definitions~\ref{def_disc_sphere} and \ref{def_disc_square}.
\end{proof}

\begin{remark}
Note that $\sin (\arccos (1-2x)) = \sqrt{x-x^2}$, whence $\bsz_n= (z_{n,1}, z_{n,2}, z_{n,3})$ where
\begin{align*}
z_{n,1} &=  2 \cos (2\pi x_{n,1}) \sqrt{x_{n,2}-x_{n,2}^2}, \\
z_{n,2} &=  2 \sin (2\pi x_{n,1}) \sqrt{x_{n,2} - x_{n,2}^2}, \\
z_{n,3} &= 1- 2x_{n,2}
\end{align*}
for $0 \le n < N$. To shorten the notation we define
\begin{eqnarray*}
\Phi(x_1,x_2) & = & T\left(x_1, \frac{\arccos (1-2x_2)}{\pi} \right) \\ & = & \left(2\cos(2\pi x_1) \sqrt{x_2-x_2^2}, 2\sin (2\pi x_1) \sqrt{x_2-x_2^2}, 1-2x_2 \right)
\end{eqnarray*}
for $(x_1,x_2) \in [0,1]^2$. Hence, $\bsz_n = \Phi(\bsx_n)$ for $0 \le n < b^m$.

Further, for $J \subseteq [0,1]^2$ we set
\begin{equation*}
\Phi(J) = \{\Phi(\bsx) \in \mathbb{S}^2: \bsx \in J\}.
\end{equation*}
\end{remark}

\subsection{Digital nets on the sphere}

The previous lemma can now be used to obtain a bound on the spherical rectangle discrepancy of $(0,m,2)$-nets lifted to the sphere via the transformation $\Phi$, which we state in the following result.

\begin{theorem}\label{thm_construction}
Let $P_N = \{\bsx_0,\ldots, \bsx_{b^m-1}\} \subseteq
(0,1)^2$ be a $(0,m,2)$-net in base $b$. Let $\bsx_n =
(x_{n,1}, x_{n,2})$ for $0 \le n < b^m$. Let $\bsz_n = \Phi(\bsx_n)$ for $0 \le n < b^m$ and $Z_N = \{\bsz_0,\ldots, \bsz_{b^m-1}\}$. Then
\begin{equation*}
D_{b^m}(Z_N;\mathbb{S}^2, \Omega) \le \frac{b^2}{b+1} \frac{m}{b^m} + \frac{1}{b^m} \left(9 + \frac{1}{b}\right) + \frac{1}{b^{2m}} \left(2b - 1 - \frac{4b+3}{(b+1)^2}\right)
\end{equation*}
and
\begin{equation*}
D^\ast_{b^m}(Z_N;\mathbb{S}^2, \Omega^\ast) \le \frac{b^2}{b+1} \frac{m}{4b^m} + \frac{1}{b^m} \left(\frac{9}{4} + \frac{1}{b}\right) + \frac{1}{b^{2m}} \left(\frac{b}{2} - \frac{1}{4} - \frac{4b+3}{4(b+1)^2}\right).
\end{equation*}
\end{theorem}

\begin{proof}
The result follows from Lemma~\ref{lem_disc} and \cite[Theorem~1 and
Remark~2]{DiKr2006}.
\end{proof}

Lemma~\ref{lem_disc} and the result of Roth~\cite{Ro1954} imply that
\begin{equation*}
D_N(Z_N; \mathbb{S}^2, \Omega) \ge D^\ast_N(Z_N;
\mathbb{S}^2, \Omega^\ast) \ge \frac{\lfloor \log_2 N \rfloor +
3}{2^8 N}
\end{equation*}
for all point sets $Z_N = \{\bsz_0,\ldots, \bsz_{N-1}\}
\subseteq \mathbb{S}^2$ (where $\log_2$ denotes the logarithm in base
$2$). Thus the construction of Theorem~\ref{thm_construction} is
optimal with respect to the discrepancy defined in
Definition~\ref{def_disc_sphere}.

Theorem~\ref{thm_construction} is not surprising since the transformation $\Phi$ is area preserving for all rectangles, which follows from \eqref{eq_area}.

\subsection{Uniform distribution on the sphere}

A sequence of $N$-point systems $\{ Z_N \}_{N \geq 2}$ on the unit sphere $\mathbb{S}^2$ in $\RR^{3}$ is said to be {\em asymptotically uniformly distributed on $\mathbb{S}^2$} if the exact integral $I(f) = \int_{\mathbb{S}^2} f \dd \sigma_2$ of any continuous function $f$ on $\mathbb{S}^2$ integrated with respect to the normalized surface area measure $\sigma_2$ on $\mathbb{S}^2$ can be approximated arbitrarily close (as $N$ becomes large) by means of the equally weighted quadrature rules $Q_N$ having these  $Z_N$ as node sets; that is
\begin{equation*}
\lim_{N \to \infty} Q_N(f) = I(f) \qquad \text{for every $f \in C(\mathbb{S}^2)$.} 
\end{equation*}
This limit relation also states that the weak-* limit (which is
defined in such a way) of the discrete measure placing a point mass
$1/N$ at every point in $Z_N$ is given by the natural measure on
$\mathbb{S}^2$. In other words, the limit distribution (as $N\to
\infty$) of the $N$-point configurations $Z_N$ is given by
$\sigma_2$. Let $|Z_N|$ denote the cardinality of the (finite) set
$Z_N$. Therefore, an equivalent characterization is that
\begin{equation*}
\lim_{N \to \infty} \left| Z_N \cap A \right| / N = \sigma_2( A )
\end{equation*}
for every $\sigma_2$-measurable set $A \subseteq \mathbb{S}^2$ (whose boundary has $2$-dimensional Hausdorff measure $0$). Informally speaking, any such set $A$ gets a fair share of points as $N$ becomes large. In fact, it is sufficient to consider spherical caps on $\mathbb{S}^2$. Let $\mathcal{C} = \{C(\bsz,t): \bsz \in \mathbb{S}^2, -1 \le t \le 1\}$ denote the set of all spherical caps. Thus, a natural measure to quantify 'equidistribution' of $N$-point systems on $\mathbb{S}^2$ is the {\em spherical cap discrepancy}
\begin{equation*}
D(Z_N; \mathbb{S}^2, \mathcal{C}) \DEF \sup_{C \subseteq \mathbb{S}^2} \left| \frac{\left|Z_N \cap C \right|}{N} - \sigma_2(C) \right|,
\end{equation*}
where the supremum is extended over all spherical caps. It is well-known that the sequence $\{Z_N\}_{N\geq2}$ is asymptotically uniformly distributed on $\mathbb{S}^2$ if and only if $D(Z_N; \mathbb{S}^2, \mathcal{C}) \to 0$ as $N \to \infty$. One can show that the following assertions are equivalent (cf., for example, \cite{Br2008}):
\begin{enumerate}
\item The sequence $\{ Z_N \}_{N \geq 2}$ is asymptotically uniformly distributed on $\mathbb{S}^2$. 
\item $\lim_{N \to \infty} Q_N(f) = I(f)$ for every $f \in C(\mathbb{S}^2)$.
\item $\lim_{N\to\infty} Q_{N}(Y_{\ell,m}) = 0$ for all spherical harmonics $Y_{\ell,m}$ of degree $\ell\geq1$ from a (real) $L_{2}(\mathbb{S}^{2},\sigma_2)$-orthonormal basis $\{ Y_{\ell,m} \}$. (This is {\em Weyl's criterion} on the sphere.)
\item $\lim_{N\to\infty} D(Z_N; \mathbb{S}^2, \mathcal{C}) = 0$.
\end{enumerate}
The spherical cap discrepancy can be estimated in terms of Weyl sums by means of {\em Erd{\"o}s-Tur{\'a}n type inequalities} \footnote{Such inequalities are a generalization of Erd{\"o}s and Tur{\'a}ns result for the unit circle \cite{ErTu1948I,ErTu1948II}.} (Grabner~\cite{Gr1991}, also cf. Li and Vaaler~\cite{LiVa1999})
\begin{equation*}
D(Z_{N}; \mathbb{S}^2, \mathcal{C}) \leq \frac{c_{1}}{L+1} + \sum_{\ell=1}^{L} \left( \frac{c_{2}}{\ell} + \frac{c_{3}}{L+1} \right) \sum_{m=1}^{Z(d,\ell)}\left| \frac{1}{N} \sum_{j=0}^{N-1} Y_{\ell,m}(\PT{x}_{j}) \right|,
\end{equation*}
where $L$ is any positive integer and the positive constants $c_{1}$, $c_{2}$, and $c_{3}$ are independent of $N$, or {\em LeVeque type inequalities} (\cite{NaSuWa2010}, generalizing LeVeques result for the unit circle \cite{LeV1965})
\begin{equation*}
c_4 \left[ \sum_{\ell=1}^\infty a_\ell \sum_{m=1}^{Z(d,\ell)} \left| \frac{1}{N} \sum_{j=0}^{N-1} Y_{\ell,m}(\PT{x}_{j}) \right|^2 \right]^{1/2} \leq D(Z_{N}; \mathbb{S}^2, \mathcal{C}) \leq c_5 \left[ \sum_{\ell=1}^\infty b_\ell \sum_{m=1}^{Z(d,\ell)} \left| \frac{1}{N} \sum_{j=0}^{N-1} Y_{\ell,m}(\PT{x}_{j}) \right|^2 \right]^{1/(d+2)},
\end{equation*}
where $a_\ell \DEF \Gamma(\ell - 1/2) / \Gamma( \ell + d + 1/2) \asymp 1 / \ell^{d+1} \FED b_\ell$,
%
for some positive constants $c_4$ and $c_5$ which are independent of $N$. The integers
\begin{equation*}
Z(d,0) = 1, \qquad Z(d,\ell) = \left(2\ell+d-1\right) \frac{\Gamma(\ell+d-1)}{\Gamma(d)\Gamma(\ell+1)}
\end{equation*}
denote the number of linearly independent spherical harmonics $Y_{\ell,m}$ of degree $\ell$.

Instead of using spherical caps as test sets, one can define discrepancy with respect to spherical rectangles (as done in this paper) or, as proposed by Sj{\"o}gren~\cite{Sj1972}, with respect to the family of so-called $K$-regular test sets.
A $\sigma_2$-measurable set $A \subseteq \mathbb{S}^2$ is defined to be {\em $K$-regular} if the $\sigma_2$-measure of the  $\delta$-neighborhood ($\delta$ sufficiently small) of its boundary,
\begin{equation*}
\left\{ \bsz \in \mathbb{S}^2 : \dist( A, \bsz ) \leq \delta, \dist( \mathbb{S}^2 \setminus A, \bsz ) \leq \delta \right\}, \quad \dist( A, \bsz ) \DEF \inf_{\bsz^\prime \in A} \left\| \bsz - \bsz^\prime \right\|,
\end{equation*}
is linearly bounded by $K \, \delta$. Clearly, spherical caps (rectangles) are $K$-regular for some $K>0$. In \cite{AnBlGo1999} Andrievskii, Blatt and G{\"o}tz related the discrepancy of a measure with the error in integration of polynomials in the following sense.

\begin{proposition} \label{prop:AnBlGo}
There exists a universal constant $C_0$ such that for every probability measure $\nu$ supported on $\mathbb{S}^2$ and every $K$-regular set $A \subseteq \mathbb{S}^2$ there holds
\begin{equation*}
\left| \nu(A) - \sigma_2(A) \right| \leq C_0 \, \inf_{n \in \NN} \left\{ \frac{K}{n} + C(\nu, n) \right\},
\end{equation*}
where
\begin{equation*}
C(\nu, n) \DEF \sup\left\{\left| \int p \dd \nu - \int p \dd \sigma_2 \right| : \begin{gathered} \text{$p$ polynomial on $\mathbb{S}^2$,}\\\text{$\deg p \leq q n$, $| p | \leq 1$ on $\mathbb{S}^2$} \end{gathered} \right\}
\end{equation*}
and $q = 3$.\footnote{Note that in \cite{AnBlGo1999} everything is done in $\RR^d$.}
\end{proposition}

In particular, if $\nu$ is the counting measure $\frac{1}{N} \sum_{\bsx \in Z_N} \delta_{\bsz}$ induced by the node set $Z_N$ of an equally weighted quadrature rule $Q_N$ on $\mathbb{S}^2$, then we have that
\begin{equation*}
\nu(A) - \sigma_2(A) = \frac{|Z_{N} \cap A|}{N} - \sigma_2(A), \qquad \int p \dd \nu - \int p \dd \sigma_2 = \frac{1}{N} \sum_{\bsz \in Z_N} p(\bsz)  - \int p \dd \sigma_2
\end{equation*}
in Proposition~\ref{prop:AnBlGo}.

Let $D(Z_N; \mathbb{S}^2, \mathcal{F}_K)$ denote the discrepancy with respect to the family of $K$-regular test sets and $D(Z_N; \mathbb{S}^2, \Omega)$ be the spherical rectangle discrepancy.

\begin{theorem} \label{thm:nets.sphere.equidistr}
Let $\{ Z_N \}$ be a sequence of $N$-point configurations with $N = b^m$, $m \geq 1$, defined by $(0,m,2)$-nets $P_N$ in base $b$ lifted to the sphere $\mathbb{S}^2$ by means of $Z_N = \Phi(P_N)$. Then the following holds:
\begin{enumerate}
\item[(i)] $\lim_{N \to \infty} D(Z_N; \mathbb{S}^2, \mathcal{F}_K) = 0$ for every fixed $K > 0$.
\item[(ii)] $\lim_{N \to \infty} D(Z_N; \mathbb{S}^2, \mathcal{F}_K ) = 0$.
\item[(iii)] $\lim_{N \to \infty} D(Z_N,\mathbb{S}^2, \Omega ) = 0$.
\end{enumerate}
Consequently, the sequence $\{Z_N\}$ is asymptotically uniformly distributed on $\mathbb{S}^2$.
\end{theorem}

\begin{proof}
Let $P_N$ be an $(0,m,2)$-net with $N = b^m$ points ($m \geq 1$) lifted to the sphere $\mathbb{S}^2$ and $\bsz_k = \Phi(\bsx_k)$ ($0 \leq k < N$). Then, using 'cylindrical' coordinates $\bsz = ( \sqrt{1 - t^2} \bsz^*, t )$, where $\bsz^* \in \mathbb{S}^1$ and $-1 \leq t \leq 1$, and the decomposition $\dd \sigma_2( \bsz ) = (1 / 2) \dd t \dd \sigma_1( \bsz^*)$ (\cite{Mu1966}), for every polynomial $p$ on $\mathbb{S}^2$ (satisfying $| p(\bsz) | \leq 1$ on $\mathbb{S}^2$) we get
\begin{equation*}
Q_N(p) - I(p) = \frac{1}{N} \sum_{k=0}^{N-1} p( \bsz_k ) - \int_{\mathbb{S}^2} p \dd \sigma_2 = \int_{\mathbb{S}^1} \left[ \frac{1}{N} \sum_{k=0}^{N-1} p( \bsz_k ) - \frac{1}{2} \int_{-1}^1 p( \bsz ) \dd t \right] \dd \sigma_1( \bsz^* ).
\end{equation*}
For each fixed $\bsz^* \in \mathbb{S}^1$ we define a new $N$-point configuration $\widehat{Z}_N$ by aligning all points in $Z_N$ such that they have the 'same $\bsz^*$', that is $\widehat{\bsz}_k = ( \sqrt{1 - t_k^2} \bsz^*, t_k )$ ($0 \leq k < N$). Clearly, $\widehat{Z}_N$ depends on $\bsz^*$. $\widehat{Z}_N$ has $N$ points, since the underlying system $P_N$ is a $(0,m,2)$-net. Indicating the dependence on $\bsz^*$ by $\widehat{\bsz}_k(\bsz^*)$ we write $Q_N(p) - I(p) = \Delta_1(p) + \Delta_2(p)$, where
\begin{align*}
\Delta_1(p) &\DEF \frac{1}{N} \sum_{k=0}^{N-1} p( \bsz_k ) - \frac{1}{N} \sum_{k=0}^{N-1} \int_{\mathbb{S}^1} p( \widehat{\bsz}_k(\bsz^*) ) \dd \sigma_1( \bsz^* ), \\
\Delta_2(p) &\DEF \int_{\mathbb{S}^1} \left[ \frac{1}{N} \sum_{k=0}^{N-1} p( \widehat{\bsz}_k(\bsz^*) ) - \frac{1}{2} \int_{-1}^1 p( \bsz ) \dd t \right] \dd \sigma_1( \bsz^* ).
\end{align*}

First, we consider $\Delta_1(p)$. By definition of $(0,m,2)$-nets every elementary interval
\begin{equation*}
\left[ a / b^d, \left( a + 1 \right) / b^d \right) \times \left[ a^\prime / b^{m-d}, \left( a^\prime + 1 \right) / b^{m-d} \right), \qquad 0 \leq a < b^d, \, 0 \leq a^\prime < b^{m - d},
\end{equation*}
contains exactly one point of $P_N$. By construction of $Z_N$, the $2$-sphere can be partitioned into $M \asymp \sqrt{N}$ polar zones $A_1, \dots, A_M$ of equal sizes $1 / \sqrt{N}$, each $A_\ell$ containing precisely $N_\ell \asymp \sqrt{N}$ points of $Z_N$ ($N_1 + \cdots + N_M = N$).
Picking $M$ heights $\tau_1, \dots, \tau_M$ such that the circle $C_\ell$ with height $\tau_\ell$ lies in $A_\ell$, we move the points in $A_\ell$ on the circle $C_\ell$ giving them the same height $\tau_\ell$ without changing the 'longitudes', that is $\bsz_k^\prime \DEF (\sqrt{1-\tau_\ell^2} \bsz_k^*, \tau_\ell)$ for $\bsz_k = (\sqrt{1-t_k^2} \bsz_k^*, t_k) \in A_\ell$. The error introduced in this way can be made arbitrarily small by increasing $N$, since the polynomial $p$ is bounded and uniformly continuous on $\mathbb{S}^2$ and the widths of the polar zones uniformly decrease with increasing $N$.
The integral $\int_{\mathbb{S}^1} p( \widehat{\bsz}_k(\bsz^*) ) \dd \sigma_1( \bsz^* )$ averages the polynomial $p$ over the circle on $\mathbb{S}^2$ with height $t_k$.
In the zone $A_\ell$ we may use the approximation $\int_{\mathbb{S}^1} p( (\sqrt{1-\tau_\ell^2} \bsz^*, \tau_\ell) ) \dd \sigma_1( \bsz^* )$.
Therefore
\begin{equation*}
\Delta_1(p) = \sum_{\ell=1}^M \frac{N_\ell}{N} \left[ \frac{1}{N_\ell} \sum_{\bsz_k \in A_\ell} p( \bsz_k^\prime ) - \int_{\mathbb{S}^1} p( (\sqrt{1-\tau_\ell^2} \bsz^*, \tau_\ell) ) \dd \sigma_1( \bsz^* )   \right] + R_1(p),
\end{equation*}
where the error $R_1(p)$ goes to zero as $N\to \infty$. Observe, that the square-bracketed expression is the error of integration of an equally weighted quadrature rule on the circle $C_\ell$ with integration points $\bsz_k^\prime$ for $\bsz_k \in A_\ell$ induced by a horizontal strip of the underlaying $(0,m,2)$-net $P_N$ for the polynomial $p$ restricted to $C_\ell$. Since the extreme discrepancy of any $(0,m,2)$-net tends to zero as $m \to \infty$, it follows that the limit distribution of the integration nodes is given by $\sigma_1$. Therefore the square-bracketed expression tends to zero uniformly for all $1 \leq \ell \leq M$ as $N\to\infty$.




Next, we consider $\Delta_2(p)$. We define the function $f(\bsz^*; t) \DEF p( (\sqrt{1-t^2} \bsz^*, t) )$ for $| t | \leq 1$ and each fixed $\bsz^* \in \mathbb{S}^1$, which is uniformly continuous and bounded in $t$ and $\bsz^*$. Thus
\begin{equation*}
\Delta_2(p) = \int_{\mathbb{S}^1} \left[ \frac{1}{N} \sum_{k=0}^{N-1} f( \bsz^*; t_k ) - \frac{1}{2} \int_{-1}^1 f( \bsz^*; t ) \dd t \right] \dd \sigma_1( \bsz^* ).
\end{equation*}
The square-bracketed expression is the error of integration for an equally weighted quadrature rule with integration nodes given by $t_k = 1 - 2 x_k$, where the $x_k$'s are the $2$nd-coordinates in the $(0,m,2)$-net $P_N$, integrating a function from the class $\{ f( \bsz^*; \cdot ) : \bsz^* \in \mathbb{S}^1 \}$. We can apply Koksma's inequality and use that $| f( \bsz^*; t ) | \leq 1$ to see that $\Delta_2(p) \to 0$ as $N \to \infty$.
Thus, for every polynomial $p$ on $\mathbb{S}^2$ (satisfying $| p(\bsz) | \leq 1$ on $\mathbb{S}^2$) we get that $| \Delta(p) | \leq | \Delta_1(p) | + | \Delta_2(p) | \to 0$ as $N \to \infty$.

Now, we can apply Proposition~\ref{prop:AnBlGo}. Let $q = 3$. Set $\nu_N = \frac{1}{N} \sum_{\bsz \in Z_N} \delta_{\bsz}$. Then to every $\varepsilon > 0$ we can chose a smallest integer $n$ such that $K / n < \varepsilon / 2$ and find an $N_n$ such that $C(2,\nu_N,n) < \varepsilon / 2$ for $N \geq N_n$ yielding
\begin{equation*}
\left| \frac{| Z_N \cap A |}{N} - \sigma_2(A) \right| \leq C_0 \left\{ \frac{K}{n} + C(2,\nu_N,n) \right\} \leq C_0 \, \varepsilon,
\end{equation*}
which holds for every $K$-regular test set $A$ ($K$ is fixed). Item~(i) follows.

Since spherical caps (rectangles) are $K$-regular ($K^\prime$-regular) for some $K, K^\prime > 0$, the Items~(ii) and (iii) follow. By the list of equivalent characterizations of uniform distribution on $\mathbb{S}^d$, the sequence $\{Z_N\}$ is asymptotically uniformly distributed on $\mathbb{S}^2$.
\end{proof}

Concerning the convergence rate of the spherical cap discrepancy, it was an observation of Beck~\cite{Be1984} that to any $N$-point set $Z_{N}$ on $\mathbb{S}^d$ there exists a spherical cap $C$ such that
\begin{equation*}
c_{1} N^{-3/4} < \left| \frac{|Z_{N} \cap C|}{N} - \sigma_2(C) \right|
\end{equation*}
and (using probability arguments) there exist $N$-point sets $Z_{N}$ on $\mathbb{S}^2$ such that
\begin{equation*}
\left| \frac{| Z_{N} \cap C|}{N} - \sigma_2(C)\right| < c_{2} N^{-3/4} \, \sqrt{\log N}
\end{equation*}
for any spherical cap $C$. (The numbers $c_{i}>0$ are constants independent of $N$.)

\begin{remark}
Thus, the correct order of decay (up to a $\sqrt{\log N}$ factor) of the spherical cap discrepancy of a sequence of low discrepancy $N$-point configurations on $\mathbb{S}^2$ is given by $N^{-3/4}$ as $N \to \infty$.
\end{remark}
This is in contrast to the convergence rates (essentially [up to a logarithmic factor] $1/N$) given in Theorem~\ref{thm_construction} for the discrepancy with respect to spherical rectangles. 

\section{Numerical integration on $\mathbb{S}^2$}\label{sec:worst-case-error}


We consider now numerical integration of functions over $\mathbb{S}^2$ in some reproducing kernel Hilbert space defined on $\mathbb{S}^2$. Let  $\widehat{P}_k$, $k \in \mathbb{N}_0$, denote the Legendre polynomials. We define a reproducing kernel
\begin{equation*}
K(\bsz,\bsz^\prime) = \sum_{\ell=0}^\infty \lambda_k \widehat{P}_\ell(\langle \bsz,\bsz^\prime \rangle),
\end{equation*}
for some real numbers $\lambda_\ell \ge 0$. The rate of decay of $\lambda_\ell$ is related to the smoothness of the functions in the associated reproducing kernel Hilbert space. In particular, if $\lambda_\ell \asymp \ell^{-2s}$, then the reproducing kernel Hilbert space $\mathcal{H}^s$ consists of functions with smoothness $s$, see \cite{BrWo20xx} for more details.

In \cite[Theorem~5.1]{BrWo20xx} a choice of $\lambda_\ell$ was
introduced for which $\lambda_\ell \asymp \ell^{-3}$ and for which the
following reproducing kernel can be written explicitly as
\begin{equation*}
K(\bsz,\bsz^\prime) = \sum_{\ell=0}^\infty \lambda_\ell \widehat{P}_\ell(\langle \bsz,\bsz^\prime \rangle) = 2 \mathcal{I} - \|\bsz-\bsz^\prime \|,
\end{equation*}
where $\mathcal{I}$ is the distance integral
\begin{equation*}
\mathcal{I} = \int_{\mathbb{S}^2} \int_{\mathbb{S}^2} \|\bsz-\bsz^\prime \| \, \mathrm{d} \sigma_2(\bsz) \mathrm{d} \sigma_2(\bsz^\prime).
\end{equation*}
This integral can be computed and has value $\mathcal{I} = 4/3$ (see \cite{BrWo20xx}). The coefficients $\lambda_\ell$ are given by 
\begin{equation*}
\lambda_0 = \mathcal{I}, \qquad \lambda_\ell = \mathcal{I} \frac{-(-1/2)_\ell}{(3/2)_\ell} \quad \text{for $\ell \ge 1$,}
\end{equation*}
where $(a)_\ell = a (a+1) \cdots (a+\ell-1) = \Gamma(a+\ell) / \Gamma(a)$ is the Pochhammer symbol. See \cite{BrWo20xx} for more details.

The function $K(\bsz,\bsz^\prime)$ is a reproducing kernel of a Hilbert space $\mathcal{H}^{3/2}$ over $\mathbb{S}^2$. In \cite{BrWo20xx} it is shown that the squared worst-case error for numerical integration in this $\mathcal{H}^{3/2}$ is given by
\begin{equation} \label{eq_wce_32}
e^2(Q_N, \mathcal{H}^{3/2}) = \mathcal{I} - \frac{1}{N^2} \sum_{k,\ell=0}^{N-1} \|\bsz_k - \bsz_\ell\|.
\end{equation}
For quadrature rules $Q_N$ with node sets maximizing the sum of distances $\sum_{k,\ell=0}^{N-1} \|\bsz_k - \bsz_\ell\|$ (that is, minimizing the right-hand side in \eqref{eq_wce_32} which in turn means that the nodes have minimal $L_2$-discrepancy), it is known that there are constants $c, C > 0$ such that
\begin{equation} \label{eq:bounds.e.squared}
c \, N^{-3/2} \le e^2(Q_N, \mathcal{H}^{3/2}) \le C \, N^{-3/2}
\end{equation}
for $N$ sufficiently large, see again \cite[Corollary~5.2]{BrWo20xx}. Note that the lower estimate always holds.

\subsection{A lower bound on the worst-case error}
\label{subsec:lower.bound}

Let $Q_N^*$ be a quadrature rule whose integration points $\bsz_1^*, \dots, \bsz_N^*$ maximize the sum of distances $\sum_{k,\ell=0}^{N-1} \|\bsz_k - \bsz_\ell\|$ or, equivalently, have minimal $L_2$-discrepancy. (Such points always exists by the continuity of the distance function and compactness of the sphere.) Then, using the lower bound in \eqref{eq:bounds.e.squared}, we have
\begin{equation*}
e^2(Q_N,\mathcal{H}^{3/2}) = \mathcal{I} - \frac{1}{N^2} \sum_{k,\ell=0}^{N-1} \left\| \bsz_k - \bsz_\ell \right \| \geq \mathcal{I} - \frac{1}{N^2} \sum_{k,\ell=0}^{N-1} \left\| \bsz_k^* - \bsz_\ell^* \right \| = e^2(Q_N^*,\mathcal{H}^{3/2}) \geq c \, N^{-3/2}
\end{equation*}
for some positive constant $c$ not depending on $N$ and $N$ sufficiently large. Thus, one obtains that $e^2(Q_N,\mathcal{H}^{3/2}) \geq c \, N^{-3/2}$ for sufficiently large $N$ for equal weight quadrature rule $Q_N$ induced by any $N$-point configuration. This is in agreement with the results obtained in \cite{HeSl2005b}.

\subsection{An upper bound on the worst-case error}

By Stolarsky's invariance principle \eqref{eq:stolarsky.inv.prncpl}
and representation \eqref{eq_wce_32} we have
\begin{equation} \label{eq:wce.l2.discr}
e^2(Q_N,\mathcal{H}^{3/2}) = \mathcal{I} - \frac{1}{N^2}
\sum_{k,\ell=0}^{N-1} \left\|\bsz_k-\bsz_\ell \right\| = 4 \left[
D_{2}( Z_N; \mathbb{S}^2, \mathcal{C} ) \right]^2.
\end{equation}
Since the spherical cap discrepancy provides an upper bound for the $L_2$-discrepancy via $D_2( Z_N; \mathbb{S}^2, \mathcal{C} ) \leq \sqrt{2} D(Z_N; \mathbb{S}^2, \mathcal{C})$, one has necessarily $e^2(Q_N,\mathcal{H}^{3/2}) \to 0$ as $N \to \infty$ for a sequence of asymptotically uniformly distributed $N$-point configurations on $\mathbb{S}^2$. This applies in particular to digital nets on the sphere of the type considered in Theorem~\ref{thm:nets.sphere.equidistr}. A natural question is if such nets yield the same order of convergence as the worst-case error over the unit ball in $\mathcal{H}^{3/2}$ over $\mathbb{S}^2$ provided with the reproducing kernel $K(\bsx, \bsy) = (8 / 3) - \| \bsx - \bsy \|$ has (which is the same as of the optimal $L_2$-discrepancy on $\mathbb{S}^2$ (see \eqref{eq:wce.l2.discr})) as suggested by the numerical results in Section~\ref{subsec:numerical.results}.

\begin{theorem} \label{thm:upper.bound}
Let $\{ Z_N \}$ be a sequence of $N$-point configurations on $\mathbb{S}^2$ defined by point sets $P_N \subseteq [0,1]^2$ lifted to the sphere $\mathbb{S}^2$ by means of $Z_N = \Phi(P_N)$. Then the equal weight quadrature rule $Q_N$ associated with $Z_N$ satisfies
\begin{equation*}
e^2(Q_N,\mathcal{H}^{3/2}) \le \left( \frac{24}{\sqrt{3}} + 2 \sqrt{2} \right) D^\ast(P_N; [0,1)^2, \mathcal{R}^\ast).
\end{equation*}
\end{theorem}

\begin{proof}
Notice that $\int_{\mathbb{S}^2} \| \bsz - \bsz^\prime \| \, \dd \sigma_2( \bsz^\prime)$ does not depend on $\bsz \in \mathbb{S}^2$. Thus, defining the function
\begin{equation} \label{eq:function}
\Delta( \bsw ) \DEF \int_{\mathbb{S}^2} \left\| \bsw - \bsz \right\| \dd \sigma_2( \bsz ) - \frac{1}{N} \sum_{k = 0}^{N-1} \left\| \bsw - \bsz_k \right\|, \qquad \bsw \in \mathbb{S}^2,
\end{equation}
one obtains
\begin{equation*} 
\frac{1}{N} \sum_{\ell = 0}^{N-1} \Delta( \bsz_\ell ) = \int_{\mathbb{S}^2} \int_{\mathbb{S}^2} \left\| \bsz - \bsz^\prime \right\| \dd \sigma_2( \bsz) \dd \sigma_2( \bsz^\prime) - \frac{1}{N^2} \sum_{k, \ell = 0}^{N-1} \left\| \bsz_k - \bsz_\ell \right\| = e^2(Q_N,\mathcal{H}^{3/2})
\end{equation*}
and therefore
\begin{equation*}
e^2(Q_N,\mathcal{H}^{3/2}) \leq \max_{0 \leq \ell < N} \Delta(\bsz_\ell) \leq \sup_{\bsz \in \mathbb{S}^2} \Delta(\bsz).
\end{equation*}
Note that the sum in \eqref{eq:function} can be seen as the potential (with respect to the distance kernel) of the discrete measure associated with $Z_N$ at $\bsw$. Thus, $\Delta( \bsz_\ell )$ gives the deviation from the leading term of the asymptotic expansion (as $N \to \infty$) of this potential at $\bsz_\ell$.

Let
\begin{equation*} 
g(\alpha,\tau) = \|\boldsymbol{w} - \Phi(\alpha,\tau) \|.
\end{equation*}
Since, by \eqref{eq_area}, the transformation $\Phi$ is area-preserving, it follows that $$\int_{\mathbb{S}^2} \|\boldsymbol{w} - \boldsymbol{z}\| \,\mathrm{d} \sigma_2(\boldsymbol{z}) = \int_{0}^1 \int_0^1 g(\alpha,\tau) \,\mathrm{d} \alpha \,\mathrm{d} \tau.$$ Thus, for any fixed $\boldsymbol{w} \in \mathbb{S}^2$, we can view $\Delta(\boldsymbol{w})$ as the integration error of integrating the function $g$ using the points $P_N$. By the Koksma-Hlawka inequality, this integration error is bounded by the star-discrepancy $D^\ast(P_N; [0,1]^2, \mathcal{R}^\ast)$ of $P_N$ times the total variation of the function $g$ in the sense of Hardy and Krause. 

The variation of $g$ in the sense of Vitali is given by
\begin{equation*}
V^{(1,2)}(g) = \sup_{\mathcal{P}} \sum_{J \in \mathcal{P}} |\delta(g,J)|,
\end{equation*}
where the supremum is taken over all partitions of $[0,1)^2$ into subintervals $J = [a_1,a_2) \times [b_1,b_2)$ and $\delta(g,J) = g(a_1,b_1)-g(a_1,b_2)-g(a_2,b_1)+g(a_2,b_2)$. The variation of the one-dimensional projections is given by
\begin{equation*}
V^{(1)}(g) = \sup_{0=x_0 < x_1 < \cdots < x_M =1} \sum_{k=1}^M |g(x_k,1)-g(x_{k-1},1)|
\end{equation*}
and
\begin{equation*}
V^{(2)}(g) = \sup_{0=y_0 < y_1 < \cdots < y_M =1} \sum_{k=1}^M |g(1,y_k)-g(1,y_{k-1})|.
\end{equation*}
The total variation of $g$ in the sense of Hardy and Krause is then given by 
\begin{equation*}
V(g) = V^{(1)}(g) + V^{(2)}(g) + V^{(1,2)}(g).
\end{equation*}

Let $\boldsymbol{w} = (0, 0,\pm 1)$. Then $\delta(g,J) = 0$ for all rectangles $J$ and therefore $V^{(1,2)}(g) = 0$. Further, $\Phi(\alpha,1) = (0,0,-1)$ for all $0 \leq \alpha \leq 1$ and therefore $V^{(1)}(g) = 0$. Finally, $V^{(2)}(g) = |g(1,0) - g(1,1)| = 2$ because of the monotonicity of $g(1,\cdot)$. Thus we have $V(g) = 2$.

Let $\bsw \neq (0,0, \pm 1)$. Then there are $u,v \in (0,1)^2$ such that $\boldsymbol{w} = \Phi(u,v)$. Again, we have $V^{(1)}(g) = 0$. Further, we have $V^{(2)}(g) \le 2 \sqrt{2}$, where we have equality if $u = 1$ and $v=1/2$.

Now we consider $V^{(1,2)}(g)$. The variation $V^{(1,2)}(g)$ does not change by changing $u$ because of the rotational symmetry of $g$ about the polar axis. We may therefore assume without loss of generality that $u=1/2$. First, let $v = 1/2$. The mixed derivative
\begin{equation*}
\frac{\partial^2 g}{\partial \alpha \partial \tau}(\alpha, \tau) = \frac{\partial^2 g}{\partial \tau \partial \alpha}(\alpha, \tau) = - 4 \pi \frac{\left( 1 - 2 \tau \right) \left( 1 + \sqrt{\left( 1 - \tau \right) \tau} \cos(2 \pi \alpha ) \right) \sin( 2 \pi \alpha )}{\sqrt{\left( 1 - \tau \right) \tau} \left( 2 + 4 \sqrt{\left( 1 - \tau \right) \tau} \cos( 2 \pi \alpha ) \right)^{3/2}}
\end{equation*}
is finite for $0 \leq \alpha \leq 1$ and $0 < \tau < 1$ with $\Phi(\alpha, \tau) \neq \boldsymbol{w}$. In particular, its sign does not change for $(\alpha, \tau)$ in the interior of one of the four quadrants $\widehat{T}_{1} = [0,1/2]\times [0,1/2]$, $\widehat{T}_{2} = [0,1/2]\times [1/2,1]$, $\widehat{T}_{3} = [1/2,1]\times [0,1/2]$, $\widehat{T}_{4} = [1/2,1] \times [1/2,1]$ of the square. For a subinterval $J$ with $\overline{J}$ contained in the interior of one of these quadrants, one can use
\begin{equation}\label{eq:delta.relation}
\delta( g; J ) = \int_J \frac{\partial^2 g}{\partial \alpha \partial \tau}( \boldsymbol{x} ) \dd \boldsymbol{x} = \mathrm{VOL}(J) \frac{\partial^2 g}{\partial \alpha \partial \tau}( \boldsymbol{x}_J ) \qquad \text{for some $\boldsymbol{x}_J \in \overline{J}$}
\end{equation}
to determine the sign of $\delta(g; J)$. For a subinterval $J$ which shares an upper or lower boundary with the square, the quantity $\delta(g; J)$ reduces to a difference of two function values, since $g$ remains constant for $\tau=0,1$. In this case one can use monotonicity of the function $g$ along horizontals (decreasing towards $\alpha = 1/2$). For $\overline{J}$ with $\boldsymbol{w}$ as one corner, one can still deduce if either $\delta(g; J) \leq 0$ or $\delta(g; J) \geq 0$. It suffices to consider partitions $\mathcal{P}_{\boldsymbol{w}}$ defined by horizontal and vertical lines including the lines with $\tau = 1/2$ and $\alpha = 1/2$. From our observations on the sign of $\delta(g; J)$ we conclude that the expression $\sum_{J \in P_{\bsw}} |\delta(g,J)|$ can be reduced to two summands for each quadrant (and by symmetry) 
\begin{align*}
V^{(1,2)}(g) 
&= \sum_{J \in \mathcal{P}_{\boldsymbol{w}}} \left|\delta(g,J) \right| = - 2 \big( g(0, \Delta\tau^\prime) - g(1/2, \Delta\tau^\prime) + 0 - 2 \big) + 2 \big( g(0, \Delta\tau^\prime) - g(1/2, \Delta\tau^\prime) \big) \\
&= + 2 \big( 2 - 0 + g(1/2, 1 - \Delta\tau^{\prime\prime}) - g(0, 1 - \Delta\tau^{\prime\prime}) \big) + 2 \big( g(0, 1 - \Delta\tau^{\prime\prime}) - g(1/2, 1 - \Delta\tau^{\prime\prime}) \big),
\end{align*}
where $\Delta\tau^\prime$ and $\Delta\tau^{\prime\prime}$ denote the ``heights'' of the first and last row of subintervals of $\mathcal{P}_{\boldsymbol{w}}$. 
Cancelling terms, we arrive at 
\begin{equation*}
V^{(1,2)}(g) = 8 \qquad \text{for $(u,v) = (1/2, 1/2)$.}
\end{equation*}
Therefore, for the case $v=1/2$, we can use the Koksma-Hlawka inequality to obtain
\begin{equation*}
|\Delta(\bsw)| \le D^\ast_N(P_N) V(g) \leq D^\ast_N(P_N) \left( 8 + 2 \sqrt{2} \right).
\end{equation*}

Now assume that $0 < v < 1/2$ (the case $1/2 < v < 1$ follows because of symmetry). Here, the mixed derivative is given by
\begin{equation*}
\frac{\partial^2 g}{\partial \alpha \partial \tau}(\alpha, \tau)  = \frac{\partial^2 g}{\partial \tau \partial \alpha}(\alpha, \tau) = - 16 \pi \left( 1 - v \right) v \left( 1 - 2 \tau \right) \, R( \cos(2\pi \alpha) ) \, \sin( 2 \pi \alpha ) \left[ g(\alpha, \tau) \right]^{-3},
\end{equation*}
where the linear form $R$ is given by
\begin{equation*}
R(x) = x + \frac{v - \tau + \left( 1 - 2 v \right) \left( 1 - \tau \right) \tau}{\sqrt{\left( 1 - v \right) v \left( 1 - \tau \right) \tau} \left( 1 - 2 \tau \right)}.
\end{equation*}
Hence the mixed derivative is again finite for $0 \leq \alpha \leq 1$ and $0 < \tau < 1$ with $\Phi(\alpha, \tau) \neq \boldsymbol{w}$. Furthermore, it vanishes on the lines $\alpha = 0,1$ and on the line $\alpha = 1/2$ (except when $\tau = 1/2$, where it is undefined). It does not vanish on the line $\tau = 1/2$. Finally, it vanishes on the set (cf. Figure~\ref{fig:fig.g})
\begin{equation*}
\Gamma \DEF \left\{ (\alpha, \tau) : \left( 1 - 2 \tau \right) \, R( \cos(2\pi \alpha) ) = 0, 0 \leq \alpha \leq 1, 0 < \tau < 1 \right\}.
\end{equation*}
(Note that the mixed derivative is singular everywhere along the lower and upper boundary of the unit square which are mapped to the poles of the sphere.) The set $\Gamma$ is symmetric in the sense that $(\alpha, \tau) \in \Gamma$ if and only if $(1-\alpha, \tau) \in \Gamma$. From the observation
\begin{equation*}
R( x ) \big|_{\tau = v} = x + \frac{0 + \left( 1 - 2v \right) \left( 1 - v \right) v}{\left(1-v\right) v \left( 1 - 2 v \right)} = x + 1
\end{equation*}
it follows that $\Gamma$ and the line $\tau = v$ intersect only in the point $(1/2,v)$. We observe that the following three relations are equivalent:
\begin{eqnarray*}
\left| R(x) - x \right| & \leq & 1, \\ \left( v - \tau \right) \left( v - 3 v \tau + 3 v \tau^2 - \tau^3 \right) & \leq & 0, \\ 
\tau  \geq  v \quad \text{and} \quad \tau^3 + 3 v \tau & \leq & 3 v \tau^2 + v.
\end{eqnarray*}
From the last pair we obtain the equivalence
\begin{equation*}
\left| R(x) - x \right| \leq 1 \qquad \text{if and only if} \qquad v \leq \tau \leq v + \left( 1 - v \right)^{2/3} v^{1/3} - \left( 1 - v \right)^{1/3} v^{2/3} \FED \tau_v.
\end{equation*}
We conclude that $\Gamma$ is the graph of a curve contained in the strip $v \leq \tau \leq \tau_v$ which is a function $\tau(\alpha)$ (assuming the value $\tau_v$ at $\alpha=0,1$ and $v$ at $\alpha=1/2$) and is a function $\alpha(\tau)$ when restricted to either the left or right half of the unit square. We record that $\tau_v < 1/2$ if $0<v<1/2$ and $\tau_v = 1/2$ if $v=1/2$. Moreover, the vertical line test also shows that the function $\tau(\alpha)$ is strictly monotonically decreasing towards $\alpha = 1/2$. It is also continuous and continuously differentiable as can be seen from an implicit differentiation of $R(\cos(2\pi \alpha))=0$ (the factor $(1-2\tau)$ can not be zero in our setting):
\begin{equation*}
\tau^\prime(\alpha) = - 4 \pi \sin( 2 \pi \alpha ) \frac{\left( 1 - 2 \tau \right)^2 \left( 1 - \tau \right)^2 \tau \sqrt{\left( 1 - v \right) v \left( 1 - \tau \right) \tau}}{v - 3 v \tau + 3 v \tau^2 - \tau^3 + 3 \left( \tau - v \right) \left( 1 - \tau \right) \tau}, \qquad \text{where $\tau = \tau(\alpha)$.}
\end{equation*}
(The denominator is a strictly monotonically increasing function on the interval $(v,1/2)$.) Using the relation $R( \cos( 2 \pi \alpha ) ) = 0$ along $\Gamma$ in order to substitute $\cos( 2 \pi \alpha )$ in $g(\alpha, \tau)$, we obtain $g$ restricted to $\Gamma$ and its first derivative in the forms
\begin{equation*}
g_\Gamma(\alpha) \DEF g( \alpha, \tau(\alpha) ) = 2 \sqrt{\frac{\tau(\alpha) - v}{1-2v}}, \qquad 
g_\Gamma^\prime( \alpha ) 
= \frac{2}{1-2v} \frac{\tau^\prime(\alpha)}{g_\Gamma(\alpha)}.
\end{equation*}
Finally, we observe that $g$ is decreasing along horizontal lines as $\alpha$ tends to $1/2$, since 
\begin{equation*}
\frac{\partial g}{\partial \alpha}(\alpha, \tau) = - 8 \pi \sqrt{\left( 1 - v \right) v \left( 1 - \tau \right) \tau} \, \frac{\sin(2\pi \alpha)}{g(\alpha,\tau)}, \qquad (\alpha, \tau) \in [0,1]^2, (\alpha, \tau) \neq (1/2,v)
\end{equation*}
and $g$ is decreasing along vertical lines in the strip $v \leq \tau \leq 1/2$ as $\tau$ tends to $v$, since
\begin{equation*}
\frac{\partial g}{\partial \tau}(\alpha, \tau) = 2 \, \frac{H( \cos( 2 \pi \alpha ) )}{g( \alpha, \tau )}, \qquad H(x) = 1 - 2 v + \sqrt{\frac{\left( 1 - v \right) v}{\left( 1 - \tau \right) \tau}} \left( 1 - 2 \tau \right) x.
\end{equation*}

\begin{figure}[ht]
\begin{center} 
\includegraphics[scale=.08]{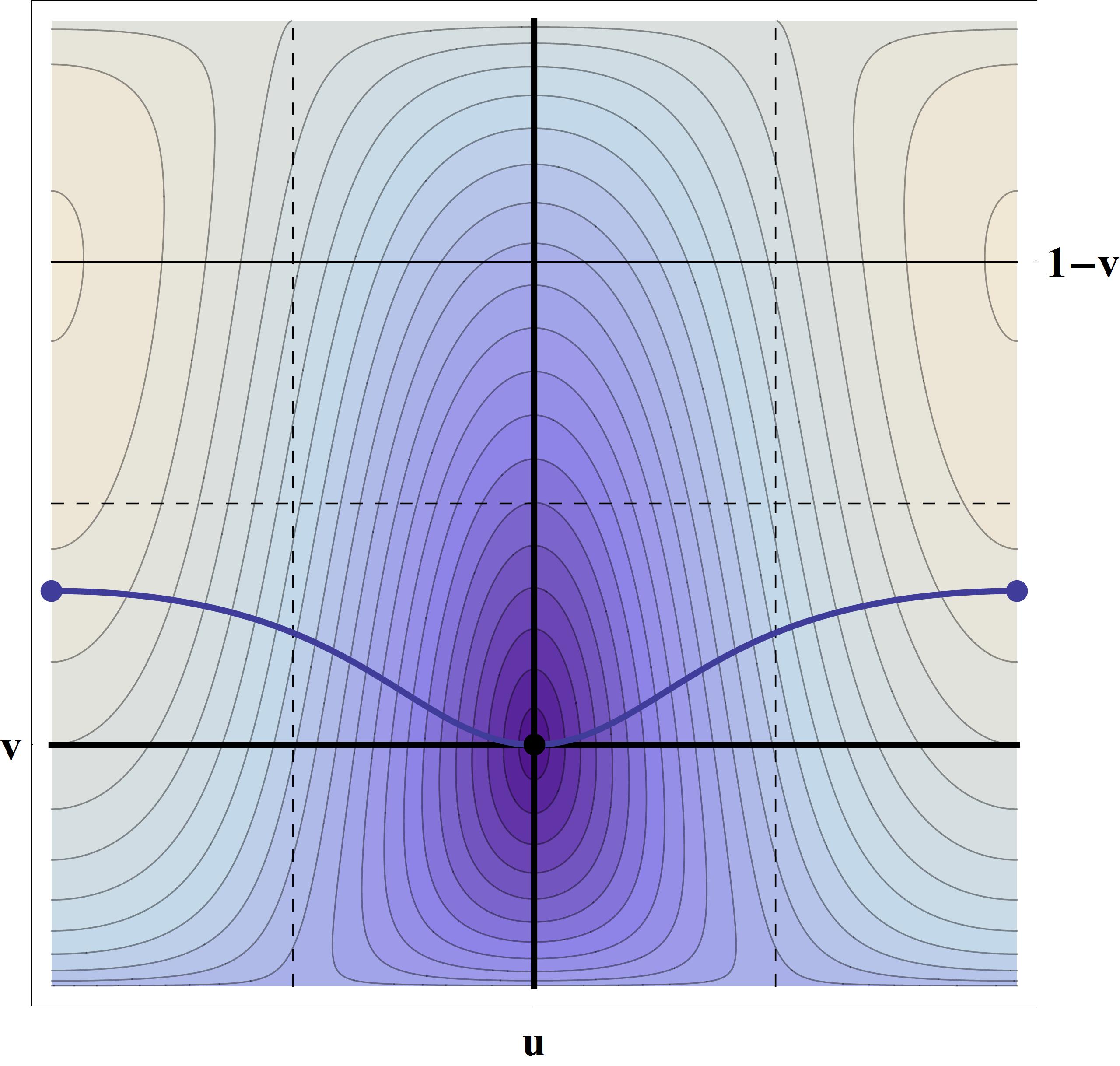}
\includegraphics[scale=.075]{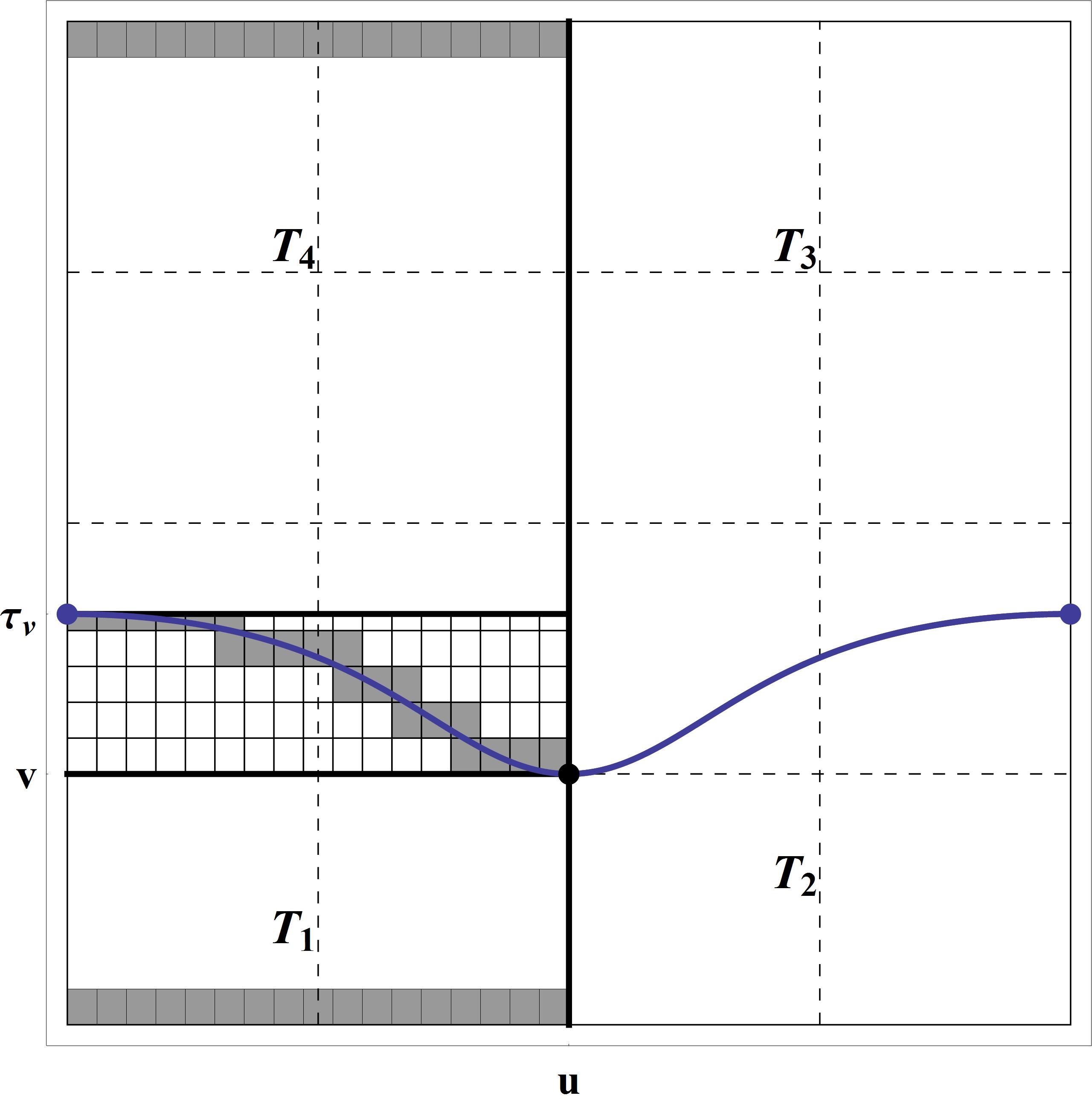}
\caption{\label{fig:fig.g} Contour plot of $g(\alpha,\tau)$ with curve $\Gamma$ and scheme to compute $V^{(1,2)}(g)$.}
\end{center}
\end{figure}

Similar as in the case $v = 1/2$, the curve $\Gamma$ and the vertical line $\alpha = 1/2$ divide the unit square into four regions $T_1, T_2, T_3, T_4$. In the interior of each region the mixed derivative has the same sign. By \eqref{eq:delta.relation}, the sign of $\delta(g; J)$ is the same for all subintervals contained in a fixed region $T_k$. Suppose $\mathcal{P}_{\boldsymbol{w}}$ is a partition of the unit square determined by a grid which is symmetric with respect to $\alpha = 1/2$ and also contains the vertical $\alpha = u ( = 1/2 )$ and the horizontals $\tau = v$ and $\tau = \tau_v$. (It suffices to consider this type of partitions.) The horizontals $\tau = v$ and $\tau = \tau_v$ subdivide the regions further into a rectangle $T_k^\prime$ without the curve $\Gamma$ and a part $T_k^{\prime\prime}$ with $\Gamma$ as part of the boundary. By symmetry (cf. Figure~\ref{fig:fig.g})
\begin{align*}
&V^{(1,2)}(g) = \sum_{J \in \mathcal{P}_{\boldsymbol{w}}} \left|\delta(g,J) \right| = 2 \left( g(0,\Delta\tau^\prime) - g(1/2,\Delta\tau^\prime) \right) - 2 \sideset{}{^\prime}{\sum}_{\substack{J \in \mathcal{P}_{\boldsymbol{w}}, \\ J \subseteq \overline{T_1^\prime}}} \delta(g,J) - 2 \sum_{\substack{J \in \mathcal{P}_{\boldsymbol{w}}, \\ J \subseteq \overline{T_1^{\prime\prime}}, \\ J \cap \Gamma = \emptyset}} \delta(g,J) \\
&\phantom{=\pm}+ 2 \sum_{\substack{J \in \mathcal{P}_{\boldsymbol{w}}, \\ J \subseteq \overline{T_1} \cup \overline{T_4}, \\ J \cap \Gamma \neq \emptyset}} \left| \delta(g,J) \right| + 2 \sum_{\substack{J \in \mathcal{P}_{\boldsymbol{w}}, \\ J \subseteq \overline{T_4^{\prime\prime}}, \\ J \cap \Gamma = \emptyset}} \delta(g,J) + 2 \sideset{}{^\prime}{\sum}_{\substack{J \in \mathcal{P}_{\boldsymbol{w}}, \\ J \subseteq \overline{T_4^\prime}}} \delta(g,J) + 2 \left( g(0,1- \Delta\tau^{\prime\prime}) - g(1/2,\Delta\tau^{\prime\prime}) \right),
\end{align*}
where $\Delta\tau^\prime$ and $\Delta\tau^{\prime\prime}$ denote the ``heights'' of the first and last row of subintervals of $\mathcal{P}_{\boldsymbol{w}}$ and in the sums $\sideset{}{^\prime}{\sum}$ one excludes the $J$'s from the first and last row. Using cancellation at ``interior'' nodes, we obtain
\begin{align*}
&V^{(1,2)}(g) 
= 2 \left( g(0,\Delta\tau^\prime) - g(1/2,\Delta\tau^\prime) \right) + 2 \left( g(0, v) - 0 - g(0, \Delta\tau^\prime) + g(1/2, \Delta\tau^\prime) \right) \\
&\phantom{=\pm}- 2 \sum_{\substack{J \in \mathcal{P}_{\boldsymbol{w}}, \\ J \subseteq \overline{T_1^{\prime\prime}}, \\ J \cap \Gamma = \emptyset}} \delta(g,J) + 2 \sum_{\substack{J \in \mathcal{P}_{\boldsymbol{w}}, \\ J \subseteq \overline{T_1} \cup \overline{T_4}, \\ J \cap \Gamma \neq \emptyset}} \left| \delta(g,J) \right| + 2 \sum_{\substack{J \in \mathcal{P}_{\boldsymbol{w}}, \\ J \subseteq \overline{T_4^{\prime\prime}}, \\ J \cap \Gamma = \emptyset}} \delta(g,J) + 2 \left( g(0,1- \Delta\tau^{\prime\prime}) - g(1/2,\Delta\tau^{\prime\prime}) \right) \\
&\phantom{=\pm}+ 2 \left( g(0, \tau_v) - g(1/2, \tau_v) + g(1/2, 1- \Delta\tau^{\prime\prime}) - g(0,1- \Delta\tau^{\prime\prime}) \right) \\
&\phantom{=}= 2 g(0, v) - 2 \sum_{\substack{J \in \mathcal{P}_{\boldsymbol{w}}, \\ J \subseteq \overline{T_1^{\prime\prime}}, \\ J \cap \Gamma = \emptyset}} \delta(g,J) + 2 \sum_{\substack{J \in \mathcal{P}_{\boldsymbol{w}}, \\ J \subseteq \overline{T_1} \cup \overline{T_4}, \\ J \cap \Gamma \neq \emptyset}} \left| \delta(g,J) \right| + 2 \sum_{\substack{J \in \mathcal{P}_{\boldsymbol{w}}, \\ J \subseteq \overline{T_4^{\prime\prime}}, \\ J \cap \Gamma = \emptyset}} \delta(g,J) + 2 \left( g(0, \tau_v) - g(1/2, \tau_v) \right).
\end{align*}
Let $v = y_m < y_{m+1} < \cdots < y_{m+K+1} = \tau_v$ denote the $\tau$-values of the grid defining $\mathcal{P}_{\boldsymbol{w}}$ in the strip $v \leq \tau \leq \tau_v$. To every $y_{m+k}$ with $0 \leq k \leq K-1$ there exists a maximal $x_{\ell_{m+k}}$ of the $\alpha$-values defining the grid such that the subinterval $[x_{\ell_{m+k-1}}, y_{m+k}) \times [x_{\ell_{m+k}}, y_{m+k+1})$ does not intersect $\Gamma$. For consistency reason we set $\ell_{m+K} = 0$, so that $x_{\ell_{m+K}} = 0$. Then
\begin{align*}
- \sum_{\substack{J \in \mathcal{P}_{\boldsymbol{w}}, \\ J \subseteq \overline{T_1^{\prime\prime}}, \\ J \cap \Gamma = \emptyset}} \delta(g,J) 
&= \sum_{k=0}^{K-1} \left( g(0,y_{m+k+1}) - g(x_{\ell_{m+k}},y_{m+k+1}) + g(x_{\ell_{m+k}},y_{m+k}) - g(0,y_{m+k}) \right) \\
&= g(0,y_{m+K}) - g(0,y_{m}) - \sum_{k=0}^{K-1} \left( g(x_{\ell_{m+k}},y_{m+k+1}) - g(x_{\ell_{m+k}},y_{m+k}) \right).
\end{align*}
Since $g$ is decreasing along horizontals in the strip $v \leq \tau \leq 1/2$ as $\tau$ goes to $v$, one gets
\begin{equation*}
- \sum_{\substack{J \in \mathcal{P}_{\boldsymbol{w}}, \\ J \subseteq \overline{T_1^{\prime\prime}}, \\ J \cap \Gamma = \emptyset}} \delta(g,J) \leq g(0,y_{m+K}) - g(0,y_{m}) \leq g(0,\tau_v) - g(0,v).
\end{equation*}
Similarly, to every $y_{m+k}$ with $1 \leq k \leq K$ there exists a minimal $x_{p_{m+k}}$ such that the subinterval $[x_{p_{m+k}}, y_{m+k}) \times [x_{p_{m+k+1}}, y_{m+k+1})$ does not intersect $\Gamma$. For consistency reason we set $p_{m}$ to that index with $x_{p_{m}} = 1/2$. Then
\begin{align*}
\sum_{\substack{J \in \mathcal{P}_{\boldsymbol{w}}, \\ J \subseteq \overline{T_4^{\prime\prime}}, \\ J \cap \Gamma = \emptyset}} \delta(g,J) 
&= \sum_{k=1}^K \left( g(x_{p_{m+k}}, y_{m+k}) - g(1/2, y_{m+k}) + g(1/2, y_{m+k+1}) - g(x_{p_{m+k}}, y_{m+k+1}) \right) \\
&= g(1/2, y_{m+K+1}) - g(1/2, y_{m+1}) - \sum_{k=1}^K \left( g(x_{p_{m+k}}, y_{m+k+1}) - g(x_{p_{m+k}}, y_{m+k}) \right) \\
&\leq g(1/2, y_{m+K+1}) - g(1/2, y_{m+1}) \leq g(1/2, y_{m+K+1}) = g(1/2, \tau_v).
\end{align*}
As an intermediate result we have
\begin{align*}
V^{(1,2)}(g) 
&\leq 2 g(0, v) + 2 \left( g(0,\tau_v) - g(0,v) \right) + 2 \sum_{\substack{J \in \mathcal{P}_{\boldsymbol{w}}, \\ J \subseteq \overline{T_1} \cup \overline{T_4}, \\ J \cap \Gamma \neq \emptyset}} \left| \delta(g,J) \right| + 2 g(1/2, \tau_v) \\
&\phantom{=}+ 2 \left( g(0, \tau_v) - g(1/2, \tau_v) \right) = 4 g(0,\tau_v) + 2 \sum_{\substack{J \in \mathcal{P}_{\boldsymbol{w}}, \\ J \subseteq \overline{T_1} \cup \overline{T_4}, \\ J \cap \Gamma \neq \emptyset}} \left| \delta(g,J) \right|.
\end{align*}
Using triangle inequality and monotonicity along horizontals we get 
\begin{align*}
&\sum_{\substack{J \in \mathcal{P}_{\boldsymbol{w}}, \\ J \subseteq \overline{T_1} \cup \overline{T_4}, \\ J \cap \Gamma \neq \emptyset}} \left| \delta(g,J) \right| 
= \sum_{k=0}^{K} \sum_{j=\ell_{m+k}}^{p_{m+k}-1} \left| g(x_j,y_{m+k}) - g(x_{j+1},y_{m+k}) + g(x_{j+1},y_{m+k+1}) - g(x_{j},y_{m+k+1}) \right| \\
&\phantom{=}\leq \sum_{k=0}^{K} \sum_{j=\ell_{m+k}}^{p_{m+k}-1} \left( g(x_j,y_{m+k}) - g(x_{j+1},y_{m+k}) \right) + \sum_{k=0}^{K} \sum_{j=\ell_{m+k}}^{p_{m+k}-1} \left( g(x_{j},y_{m+k+1}) - g(x_{j+1},y_{m+k+1}) \right) \\
&\phantom{=}= \sum_{k=0}^{K} \left( g(x_{\ell_{m+k}},y_{m+k}) - g(x_{p_{m+k}},y_{m+k}) \right) + \sum_{k=0}^{K} \left( g(x_{\ell_{m+k}},y_{m+k+1}) - g(x_{p_{m+k}},y_{m+k+1}) \right) \\
&\phantom{=}\leq 2 \sum_{k=0}^{K} \left( g_\Gamma(x_{\ell_{m+k}}) - g_\Gamma(x_{p_{m+k}}) \right).
\end{align*}
In the last step we used that the left side of a horizontal segment of successive rectangles of the covering of $\Gamma$ is below and the right side is above the curve $\Gamma$ by construction, cf. Figure~\ref{fig:fig.g}. Recall, that $g_\Gamma$ is strictly monotonically decreasing as $\alpha$ tends to $1/2$. Since $g_\Gamma(x_{\ell_{m+k-1}}) > g_\Gamma(x_{p_{m+k}})$ for $0 \leq k \leq K - 1$ by monotonicity of $g_\Gamma$, we have
\begin{align*}
&\sum_{\substack{J \in \mathcal{P}_{\boldsymbol{w}}, \\ J \subseteq \overline{T_1} \cup \overline{T_4}, \\ J \cap \Gamma \neq \emptyset}} \left| \delta(g,J) \right|
\leq 2 \Big\{ \left( g_\Gamma(x_{\ell_{m}}) - g_\Gamma(x_{p_{m}}) \right) + \sum_{k=1}^{K} \left( g_\Gamma(x_{\ell_{m+k}}) - g_\Gamma(x_{\ell_{m+k-1}}) \right) \\
&\phantom{=}+ \sum_{k=1}^{K} \left( g_\Gamma(x_{\ell_{m+k-1}}) - g_\Gamma(x_{p_{m+k}}) \right) \Big\} = 2 \Big\{ g_\Gamma(0) - g_\Gamma(1/2) + \sum_{k=1}^{K} \left( g_\Gamma(x_{\ell_{m+k-1}}) - g_\Gamma(x_{p_{m+k}}) \right) \Big\}.
\end{align*}
By construction $x_{\ell_{m+K-1}} > x_{p_{m+K}} > x_{\ell_{m+K-2}} > x_{p_{m+K-1}} > \cdots > x_{\ell_{m+k}} > x_{p_{m}} = 1/2$. So, because of monotonicity of $g_\Gamma$, we increase the right-hand side of the estimate above when including terms terms for the gaps $[x_{p_{m+k}},x_{\ell_{m+k-2}}]$. Thus
\begin{equation*}
\sum_{\substack{J \in \mathcal{P}_{\boldsymbol{w}}, \\ J \subseteq \overline{T_1} \cup \overline{T_4}, \\ J \cap \Gamma \neq \emptyset}} \left| \delta(g,J) \right| \leq 4 \left( g_\Gamma(0) - g_\Gamma(1/2) \right) = 4 g_\Gamma(0),
\end{equation*}
where the expression in parentheses is the total variation of (monotone) one-dimensional function $g_\Gamma$ restricted to $[0,1/2]$. Putting everything together, we obtain
\begin{equation*}
V^{(1,2)}(g) \leq 4 g(0,\tau_v) + 2 \sum_{\substack{J \in \mathcal{P}_{\boldsymbol{w}}, \\ J \subseteq \overline{T_1} \cup \overline{T_4}, \\ J \cap \Gamma \neq \emptyset}} \left| \delta(g,J) \right| \leq 4 g(0,\tau_v) + 8 g_\Gamma(0) = 12 g_\Gamma(0) = 24 \sqrt{ \frac{\tau_v - v}{1-2v} } 
\end{equation*}
It can be shown that the right-most side above as function of $v$ is strictly monotonically increasing on $[0,1/2)$ with
\begin{equation*}
\lim_{v \to 1/2^-} \sqrt{ \frac{\tau_v - v}{1-2v} } = \frac{1}{\sqrt{3}}.
\end{equation*}
We can use the Koksma-Hlawka inequality to obtain
\begin{equation*}
|\Delta(\bsw)| \le D^\ast(P_N; [0,1]^2, \mathcal{R}^\ast) V(g) \leq D^\ast(P_N; [0,1]^2, \mathcal{R}^\ast) \left( \frac{24}{\sqrt{3}} + 2 \sqrt{2} \right) \quad \text{for $0 < v < 1/2$.}
\end{equation*}
This concludes the proof.
\end{proof}

Digital nets \cite{DiPi2010, Ni1992} are explicit constructions of point sets in $[0,1)^2$ with $D^\ast(P_N; [0,1]^2, \mathcal{R}^\ast) = \mathcal{O}(N^{-1} \log N)$. By mapping them to the unit sphere $\mathbb{S}^2$ one also obtains explicit constructions of points on $\mathbb{S}^2$ with favorable discrepancy. We have the following corollary.

\begin{corollary}
Let $b\ge 2$ and $m \ge 1$ be integers and let $N= b^m$. Let $\{ Z_N \}$ be a sequence of $N$-point configurations on $\mathbb{S}^2$ defined by digital $(0,m,2)$-nets $P_N \in [0,1]^2$ lifted to the sphere $\mathbb{S}^2$ by means of $Z_N = \Phi(P_N)$. Then the equal weight quadrature rules $Q_N$ associated with $Z_N$ satisfy
\begin{equation*}
e^2(Q_N,\mathcal{H}^{3/2})  = \mathcal{O}(N^{-1} \log N) \qquad \text{as $N \to \infty$.}
\end{equation*}
\end{corollary}

\subsection{Numerical results}
\label{subsec:numerical.results}

We consider now the worst case error \eqref{eq_wce_32} for quadrature rules defined by digital
nets based on a Sobol' sequence, see \cite{So1967}. For efficient
implementations of Sobol' sequences see
\cite{AnSa1979,BrFo1988,JoKu2003}. Figure~\ref{fig1} shows the
squared worst-case error $e^2(Q_N, \mathcal{H}^{3/2})$.
The numerical results suggest that the squared worst case error $e^2(Q_N, \mathcal{H}^{3/2})$ converges with order $\mathcal{O}(N^{-3/2})$ as $N \to \infty$.

Note that the first point of the Sobol' sequence (or any digital
$(0,2)$-sequence for that matter) is always $(0,0)$. Since
$\Phi(0,0) = (1,0,0)$, this point gets mapped to the North Pole. On
the other hand, no point of the digital sequence gets mapped to the
South Pole. This might not be a desirable feature. To remedy this
situation one can randomize the $(0,2)$-sequence using a scrambling
algorithm, see \cite{Ma1998,Ow2003}. In this case the sequence in
$[0,1]^2$ is still a $(0,2)$-sequence, but the point $(0,0)$ only
occurs with probability $0$. Numerical investigation of scrambled
Sobol' sequences yield similar results as the one shown in
Figure~\ref{fig1}.

%

\begin{figure}[ht]
\begin{center}
\includegraphics[scale=0.55]{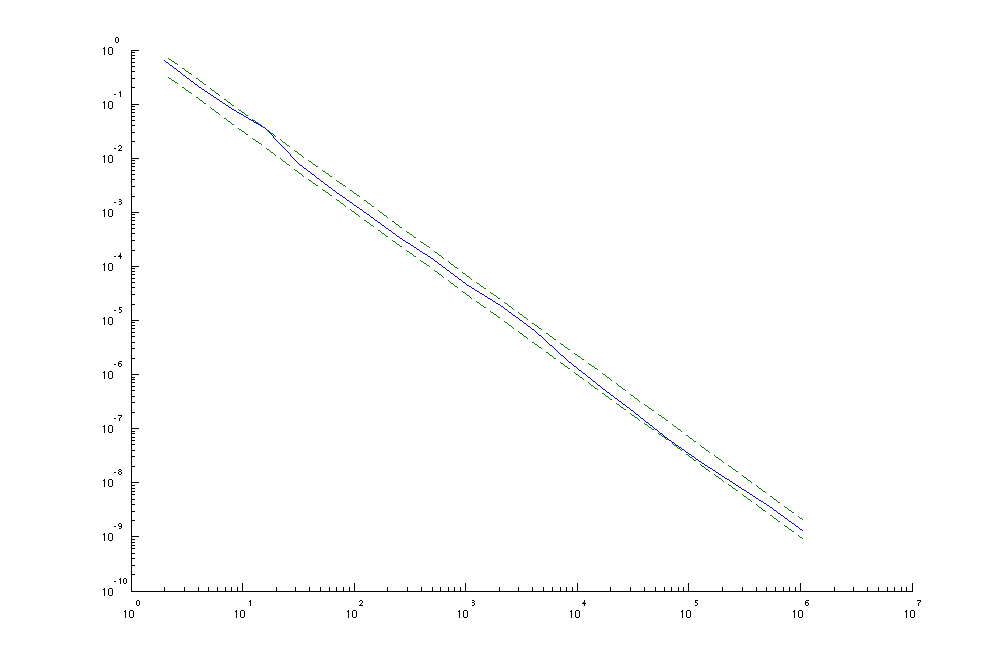}
\caption{\label{fig1} The dashed lines show $N^{-3/2}$ and $(9/4) N^{-3/2}$, and the curve shows the squared worst-case error $e^2(Q_N,\mathcal{H}^{3/2})$, where the quadrature points are a digital net mapped to the sphere.}
\end{center}
\end{figure}

\begin{table}[ht]
\begin{center}
\begin{tabular}{l|l|l|l|l}
$m$ & $N = 2^m$ & $e^2(Q_N,\mathcal{H}^{3/2})$ & $N^{-3/2}$ & $N^{3/2} e^2(Q_N,\mathcal{H}^{3/2})$  \\ \hline
1 & 2  &  6.2622e-01 & 3.5355e-01 & 1.7712 \\
2 & 4  &  2.1149e-01 & 1.2500e-01 & 1.6920 \\
3 & 8  &  8.1448e-02 & 4.4194e-02 & 1.8430 \\
4 & 16  &  3.5091e-02 & 1.5625e-02 & 2.2459 \\
5 & 32  &  8.0526e-03 & 5.5242e-03 & 1.4577 \\ \hline
6 & 64  &  2.6309e-03 & 1.9531e-03 & 1.3470 \\
7 & 128  &  9.4336e-04 & 6.9053e-04 & 1.3661 \\
8 & 256  &  3.4501e-04 & 2.4414e-04 & 1.4132 \\
9 & 512  &  1.3374e-04 & 8.6316e-05 & 1.5495 \\
10 & 1024  &  4.6029e-05 & 3.0517e-05 & 1.5083 \\ \hline
11 & 2048 &  1.8846e-05 & 1.0789e-05 & 1.7468 \\
12 & 4096 &  6.4670e-06 & 3.8146e-06 & 1.6953 \\
13 & 8192 & 1.7873e-06 & 1.3486e-06 & 1.3252 \\
14 & 16384 & 5.6815e-07 & 4.7683e-07 & 1.1915 \\
15 & 32768 & 1.9912e-07 & 1.6858e-07 & 1.1811 \\ \hline 
16 & 65536  & 6.3194e-08 & 5.9604e-08 & 1.0602 \\
17 &  131072 & 2.4122e-08 & 2.1073e-08 & 1.1447 \\
18 &262144  & 9.1906e-09 & 7.4505e-09 & 1.2335 \\ 
19 & 524288  & 3.7001e-09 & 2.6341e-09 & 1.4047 \\
20 & 1048576  & 1.3068e-09 & 9.3132e-10 & 1.4032 \\ \hline
\end{tabular}
\end{center}
\caption{Numerical results: The worst-case error obtained when using digital nets over $\mathbb{Z}_2$ lifted to the sphere $\mathbb{S}^2$. }
\label{table1}
\end{table}


The numerical results lead us to the following conjecture.

\begin{conjecture}
Let $\bsx_0,\bsx_1,\ldots \in [0,1)^2$ be a $(0,2)$-sequence and let $\bsz_0,\bsz_1,\ldots \in \mathbb{S}^2$ be the corresponding points on the sphere obtained by using the mapping $\Phi$. Let $Q_{N}(f) = \frac{1}{N} \sum_{n=0}^{N-1} f(\bsz_n)$. Then we have
\begin{equation*}
e^2(Q_N, \mathcal{H}^{3/2}) = \mathcal{O}(N^{-3/2}) \qquad \text{as $N \to \infty$.}
\end{equation*}
\end{conjecture}

In other words, a $(0,2)$-sequence lifted to the $2$-sphere via the mapping $\Phi$ achieves the optimal rate of convergence of the worst-case integration error in $\mathcal{H}^{3/2}$. By Stolarsky's invariance principle, the conjecture also implies that a $(0,2)$-sequence lifted to the $2$-sphere via the mapping $\Phi$ achieves the optimal rate of decay of the spherical cap $L_2$-discrepancy.

\vspace{10mm}
{\bf Acknowledgement:} The first author is grateful to the School of Mathematics and Statistics at UNSW for their support.

\bibliographystyle{abbrv}
\bibliography{bibliography}

\def\cprime{$'$} \def\cprime{$'$} \def\cprime{$'$}
\begin{thebibliography}{10}

\bibitem{AnBlGo1999}
V.~V. Andrievskii, H.-P. Blatt, and M.~G{\"o}tz.
\newblock Discrepancy estimates on the sphere.
\newblock {\em Monatsh. Math.}, 128(3):179--188, 1999.

\bibitem{AnSa1979}
I.~A. Antonov and V.~M. Saleev.
\newblock An effective method for the computation of {$\lambda {\rm P}_{\tau
  }$}-sequences.
\newblock {\em Zh. Vychisl. Mat. i Mat. Fiz.}, 19(1):243--245, 271, 1979.

\bibitem{BaBr2009}
B.~Ballinger, G.~Blekherman, H.~Cohn, N.~Giansiracusa, E.~Kelly, and
  A.~Sch{\"u}rmann.
\newblock Experimental study of energy-minimizing point configurations on
  spheres.
\newblock {\em Experiment. Math.}, 18(3):257--283, 2009.

\bibitem{Be1984}
J.~Beck.
\newblock Sums of distances between points on a sphere---an application of the
  theory of irregularities of distribution to discrete geometry.
\newblock {\em Mathematika}, 31(1):33--41, 1984.

\bibitem{BeCa2009}
E.~Bendito, A.~Carmona, A.~M. Encinas, J.~M. Gesto, A.~G{\'o}mez,
  C.~Mouri{\~n}o, and M.~T. S{\'a}nchez.
\newblock Computational cost of the {F}ekete problem. {I}. {T}he forces method
  on the 2-sphere.
\newblock {\em J. Comput. Phys.}, 228(9):3288--3306, 2009.

\bibitem{BeClDu2004}
M.~K. Berkenbusch, I.~Claus, C.~Dunn, L.~P. Kadanoff, M.~Nicewicz, and S.~C.
  Venkataramani.
\newblock Discrete charges on a two dimensional conductor.
\newblock {\em J. Statist. Phys.}, 116(5-6):1301--1358, 2004.

\bibitem{BoHaSa2008}
S.~V. Borodachov, D.~P. Hardin, and E.~B. Saff.
\newblock Asymptotics for discrete weighted minimal {R}iesz energy problems on
  rectifiable sets.
\newblock {\em Trans. Amer. Math. Soc.}, 360(3):1559--1580 (electronic), 2008.

\bibitem{BrFo1988}
P.~Bratley and B.~L. Fox.
\newblock Algorithm 659: Implementing sobol$'$s quasirandom sequence generator.
\newblock {\em ACM Trans. Math. Softw.}, 14:88--100, 1988.

\bibitem{Br2008}
J.~S. Brauchart.
\newblock Optimal logarithmic energy points on the unit sphere.
\newblock {\em Math. Comp.}, 77(263):1599--1613, 2008.

\bibitem{BrDi2011_pre}
J.~S. Brauchart and J.~Dick.
\newblock {A simple Proof of Stolarsky's Invariance Principle}.
\newblock arXiv:1101.4448v1 [math.NA]; submitted.

\bibitem{BrWo20xx}
J.~S. Brauchart and R.~S. Womersley.
\newblock Numerical integration over the unit sphere,
  $\mathcal{L}_2$-discrepancy and sum of distances.
\newblock manuscript, 29 pages.

\bibitem{CoKu2007}
H.~Cohn and A.~Kumar.
\newblock Universally optimal distribution of points on spheres.
\newblock {\em J. Amer. Math. Soc.}, 20(1):99--148 (electronic), 2007.

\bibitem{CuFr1997}
J.~Cui and W.~Freeden.
\newblock Equidistribution on the sphere.
\newblock {\em SIAM J. Sci. Comput.}, 18(2):595--609, 1997.

\bibitem{DiKr2006}
J.~Dick and P.~Kritzer.
\newblock A best possible upper bound on the star discrepancy of
  {$(t,m,2)$}-nets.
\newblock {\em Monte Carlo Methods Appl.}, 12(1):1--17, 2006.

\bibitem{DiPi2010}
J.~Dick and F.~Pillichshammer.
\newblock {\em {Digital Nets and Sequences. Discrepancy Theory and Quasi-Monte
  Carlo Integration}}.
\newblock Cambridge University Press, Cambridge, 2010.

\bibitem{ErTu1948I}
P.~Erd{\"o}s and P.~Tur{\'a}n.
\newblock On a problem in the theory of uniform distribution. {I}.
\newblock {\em Nederl. Akad. Wetensch., Proc.}, 51:1146--1154 = Indagationes
  Math. 10, 370--378 (1948), 1948.

\bibitem{ErTu1948II}
P.~Erd{\"o}s and P.~Tur{\'a}n.
\newblock On a problem in the theory of uniform distribution. {II}.
\newblock {\em Nederl. Akad. Wetensch., Proc.}, 51:1262--1269 = Indagationes
  Math. 10, 406--413 (1948), 1948.

\bibitem{Fa1982}
H.~Faure.
\newblock Discr\'epance de suites associ\'ees \`a un syst\`eme de num\'eration
  (en dimension {$s$}).
\newblock {\em Acta Arith.}, 41(4):337--351, 1982.

\bibitem{Gr1991}
P.~J. Grabner.
\newblock Erd{\H o}s-{T}ur\'an type discrepancy bounds.
\newblock {\em Monatsh. Math.}, 111(2):127--135, 1991.

\bibitem{HaSa2004}
D.~P. Hardin and E.~B. Saff.
\newblock Discretizing manifolds via minimum energy points.
\newblock {\em Notices Amer. Math. Soc.}, 51(10):1186--1194, 2004.

\bibitem{HaSa2005}
D.~P. Hardin and E.~B. Saff.
\newblock Minimal {R}iesz energy point configurations for rectifiable
  {$d$}-dimensional manifolds.
\newblock {\em Adv. Math.}, 193(1):174--204, 2005.

\bibitem{HeSl2005b}
K.~Hesse and I.~H. Sloan.
\newblock Optimal lower bounds for cubature error on the sphere {$S^2$}.
\newblock {\em J. Complexity}, 21(6):790--803, 2005.

\bibitem{JoKu2003}
S.~Joe and F.~Y. Kuo.
\newblock Remark on algorithm 659: Implementing sobol$'$s quasir- andom
  sequence generator.
\newblock {\em ACM Trans. Math. Softw.}, 29:49--57, 2003.

\bibitem{LeV1965}
W.~J. LeVeque.
\newblock An inequality connected with {W}eyl's criterion for uniform
  distribution.
\newblock In {\em Proc. {S}ympos. {P}ure {M}ath., {V}ol. {VIII}}, pages 22--30.
  Amer. Math. Soc., Providence, R.I., 1965.

\bibitem{LiVa1999}
X.-J. Li and J.~D. Vaaler.
\newblock Some trigonometric extremal functions and the {E}rd{\H o}s-{T}ur\'an
  type inequalities.
\newblock {\em Indiana Univ. Math. J.}, 48(1):183--236, 1999.

\bibitem{LuPhSa1986}
A.~Lubotzky, R.~Phillips, and P.~Sarnak.
\newblock Hecke operators and distributing points on the sphere. {I}.
\newblock {\em Comm. Pure Appl. Math.}, 39(S, suppl.):S149--S186, 1986.
\newblock Frontiers of the mathematical sciences: 1985 (New York, 1985).

\bibitem{Ma1998}
J.~Matou{\v{s}}ek.
\newblock On the {$L_2$}-discrepancy for anchored boxes.
\newblock {\em J. Complexity}, 14(4):527--556, 1998.

\bibitem{Mu1966}
C.~M{\"u}ller.
\newblock {\em Spherical harmonics}, volume~17 of {\em Lecture Notes in
  Mathematics}.
\newblock Springer-Verlag, Berlin, 1966.

\bibitem{NaSuWa2010}
F.~J. Narcowich, X.~Sun, J.~D. Ward, and Z.~Wu.
\newblock Leveque type inequalities and discrepancy estimates for minimal
  energy configurations on spheres.
\newblock {\em J. Approx. Theory}, 162(6):1256--1278, 2010.

\bibitem{Ni1988}
H.~Niederreiter.
\newblock Low-discrepancy and low-dispersion sequences.
\newblock {\em J. Number Theory}, 30(1):51--70, 1988.

\bibitem{Ni1992}
H.~Niederreiter.
\newblock {\em Random number generation and quasi-{M}onte {C}arlo methods},
  volume~63 of {\em CBMS-NSF Regional Conference Series in Applied
  Mathematics}.
\newblock Society for Industrial and Applied Mathematics (SIAM), Philadelphia,
  PA, 1992.

\bibitem{NiXi1996}
H.~Niederreiter and C.~Xing.
\newblock Low-discrepancy sequences and global function fields with many
  rational places.
\newblock {\em Finite Fields Appl.}, 2(3):241--273, 1996.

\bibitem{NiXi1995}
H.~Niederreiter and C.~P. Xing.
\newblock Low-discrepancy sequences obtained from algebraic function fields
  over finite fields.
\newblock {\em Acta Arith.}, 72(3):281--298, 1995.

\bibitem{Ow2003}
A.~B. Owen.
\newblock Variance with alternative scramblings of digital nets.
\newblock {\em ACM Trans. Modeling Comp. Simulation}, 12:363--378, 2003.

\bibitem{Ro1954}
K.~F. Roth.
\newblock On irregularities of distribution.
\newblock {\em Mathematika}, 1:73--79, 1954.

\bibitem{Sch1969}
W.~M. Schmidt.
\newblock Irregularities of distribution. {IV}.
\newblock {\em Invent. Math.}, 7:55--82, 1969.

\bibitem{Sj1972}
P.~Sj{\"o}gren.
\newblock Estimates of mass distributions from their potentials and energies.
\newblock {\em Ark. Mat.}, 10:59--77, 1972.

\bibitem{SlWo2004}
I.~H. Sloan and R.~S. Womersley.
\newblock Extremal systems of points and numerical integration on the sphere.
\newblock {\em Adv. Comput. Math.}, 21(1-2):107--125, 2004.

\bibitem{So1967}
I.~M. Sobol{\cprime}.
\newblock Distribution of points in a cube and approximate evaluation of
  integrals.
\newblock {\em \u Z. Vy\v cisl. Mat. i Mat. Fiz.}, 7:784--802, 1967.

\bibitem{St1973}
K.~B. Stolarsky.
\newblock Sums of distances between points on a sphere. {II}.
\newblock {\em Proc. Amer. Math. Soc.}, 41:575--582, 1973.

\bibitem{SuCh2008}
X.~Sun and Z.~Chen.
\newblock Spherical basis functions and uniform distribution of points on
  spheres.
\newblock {\em J. Approx. Theory}, 151(2):186--207, 2008.

\bibitem{Wa1990}
G.~Wagner.
\newblock On means of distances on the surface of a sphere (lower bounds).
\newblock {\em Pacific J. Math.}, 144(2):389--398, 1990.

\bibitem{Wa1992b}
G.~Wagner.
\newblock Erd{\H o}s-{T}ur\'an inequalities for distance functions on spheres.
\newblock {\em Michigan Math. J.}, 39(1):17--34, 1992.

\bibitem{Wa1992}
G.~Wagner.
\newblock On means of distances on the surface of a sphere. {II}. {U}pper
  bounds.
\newblock {\em Pacific J. Math.}, 154(2):381--396, 1992.

\bibitem{XiNi1995}
C.~P. Xing and H.~Niederreiter.
\newblock A construction of low-discrepancy sequences using global function
  fields.
\newblock {\em Acta Arith.}, 73(1):87--102, 1995.

\end{thebibliography}

\end{document}